\documentclass[12pt]{amsart}
\usepackage{amsmath,mathtools,amsthm,amsfonts,bigints,amssymb, mathrsfs}

\newtheorem{theorem}{Theorem}[section]
\newtheorem{lemma}[theorem]{Lemma}
\newtheorem{proposition}[theorem]{Proposition}
\newtheorem{corollary}[theorem]{Corollary}

\newtheorem{definition}[theorem]{Definition}

\newtheorem{remark}[theorem]{Remark}

\newcommand{\R}{\mathbb R}%
\newcommand{\C}{\mathbb C}%
\newcommand{\N}{\mathbb N}%
\numberwithin{equation}{section}

\begin{document}
\title[Superharmonic functions on harmonic manifolds]{Boundary exceptional sets for radial limits of superharmonic functions on non-positively curved harmonic manifolds of purely exponential volume growth}
\author[Utsav Dewan]{Utsav Dewan}
\address{Stat-Math Unit, Indian Statistical Institute, 203 B. T. Rd., Kolkata 700108, India}
\email{utsav\_r@isical.ac.in}

\begin{abstract} 
By classical Fatou type theorems in various setups, it is well-known that positive harmonic functions have non-tangential limit at almost every point on the boundary. In this paper, in the setting of non-positively curved Harmonic manifolds of purely exponential volume growth, we are interested in the size of the exceptional sets of points on the boundary at infinity, where a suitable function blows up faster than a prescribed growth rate, along radial geodesic rays. For Poisson integrals of complex measures, we obtain a sharp bound on the Hausdorff dimension of the exceptional sets, in terms of the mean curvature of horospheres and the parameter of the growth rate. In the case of the Green potentials, we obtain similar upper bounds and also construct Green potentials that blow up faster than a prescribed rate on lower Hausdorff dimensional realizable sets. So we get a gap in the corresponding Hausdorff dimensions due to the assumption of variable pinched non-positive sectional curvature. We also obtain a Riesz decomposition theorem for subharmonic functions. Combining the above results we get our main result concerning Hausdorff dimensions of the exceptional sets of positive superharmonic functions.
\end{abstract}
\maketitle
\tableofcontents
\section{Introduction}
The boundary behavior of harmonic functions is one of the most well-studied topics in classical potential theory. The celebrated theorem of Fatou tells us that any positive harmonic function on the unit disk admits a non-tangential limit at almost all points on the boundary. This fact was generalized to rank one Riemannian symmetric spaces of non-compact type (for admissible limits) by Kor\'anyi \cite{Koranyi}  and then to Hadamard manifolds of pinched negative curvature by Anderson and Schoen \cite{AS}. 

\medskip

Then a natural question to ask is: how does a positive harmonic function behave along radial geodesic rays, on the complement of this full measure subset of the boundary. More precisely, how quickly can a positive harmonic function grow or how large can the exceptional set (in the boundary) be where the positive harmonic function blows up faster than a prescribed rate. These are the questions which we will address in this note, in the setting of non-positively curved Harmonic manifolds of purely exponential volume growth.

\medskip

Throughout this article, all Riemannian manifolds are assumed to be complete, simply connected and of dimension $n \ge 3$. A Harmonic manifold is a Riemannian manifold $X$ such that for any point $x \in X$, there exists a non-constant harmonic function on a punctured neighbourhood of $x$ which is radial around $x$, that is, only depends on the geodesic distance from $x$. By purely exponential volume growth, we mean that there exist constants $C > 1,\: h > 0$ such that the volume of metric balls $B(x, R)$ with center $x \in X$ and radius $R>1$, satisfies the asymptotics:
\begin{equation*}
\frac{1}{C} e^{hR} \le vol(B(x,R)) \le C e^{hR} \:.
\end{equation*}
It turns out that in our case, the constant $h > 0$ agrees with the mean curvature of the horospheres. It is well-known that the sectional curvature of Harmonic manifolds are bounded below (see \cite{Besse, RR}), that is, there exists $b>0$ such that $K_X \ge -b^2$\:.

\medskip

The class of non-positively curved Harmonic manifolds of purely exponential volume growth includes all the known examples of non-compact non-flat Harmonic manifolds: the rank one Riemannian symmetric spaces of non-compact type and the Damek-Ricci spaces.
\medskip

Let $X$ be a non-positively curved Harmonic manifold of purely exponential volume growth with mean curvature of horospheres $h>0$. We fix an origin $o$ in $X$ and let $\partial X$ be the boundary at infinity. Now for a $\xi \in \partial X$, let $\gamma_\xi$ denote the unit-speed geodesic ray such that $\gamma_\xi(0)=o,\:\gamma_\xi(+\infty)=\xi$ and $d(o,\gamma_\xi(t))=t$ for all $t \in (0,+\infty)$. Then the Poisson kernel of $X$ is given by,
 \begin{equation} \label{poisson_kernel}
 P(x,\xi) = e^{-hB_{\xi}(x)} \:,\text{ for all } x \in X,\: \xi \in \partial X,
 \end{equation}
 where $B_{\xi}(x)$ is the Busemann function, defined by
 \begin{equation} \label{busemann}
 B_{\xi}(x) = \displaystyle\lim_{t \to \infty} \left(d\left(x,\gamma_\xi(t)\right) - d\left(o,\gamma_\xi(t)\right)\right) \:.
 \end{equation} 

\medskip

The Martin representation formula \cite[Corollary 5.13]{KL} asserts that the positive harmonic functions on $X$ are given by Poisson  integrals of finite, positive Borel measures on $\partial X$. More generally, for a complex measure $\mu$ on $\partial X$, let $u=P[\mu]$ be the Poisson integral of $\mu$. Then for any $\xi \in \partial X$ and any $t \in (0,+\infty)$, by (\ref{poisson_kernel}), (\ref{busemann}) and the triangle inequality, we have
\begin{equation} \label{bounding_poisson}
|u(\gamma_\xi(t))| =\left|\int_{\partial X} P(\gamma_\xi(t),\eta)\:d\mu(\eta)\right| \le e^{ht}\: |\mu|(\partial X)\:,
\end{equation}
where $|\mu|(\partial X)$ is the total variation of $\mu$.

\medskip

Then (\ref{bounding_poisson}) motivates us to consider for $\beta \in [0,h]$ and a complex-valued function $u$ on $X$, the following sets
\begin{equation} \label{Ebeta}
E_\beta(u) := \left\{\xi \in \partial X : \displaystyle\limsup_{t \to +\infty} e^{-\beta t} \left|u\left(\gamma_{\xi}(t)\right)\right| > 0\right\}
\end{equation}
and
\begin{equation} \label{Ebetainf}
E^\infty_\beta(u) := \left\{\xi \in \partial X : \displaystyle\limsup_{t \to +\infty} e^{-\beta t} \left|u\left(\gamma_{\xi}(t)\right)\right| =+\infty\right\} \:.
\end{equation}

\medskip

When $K_X \le -1$, there is a natural metric called the visual metric, denoted by $\rho$ on $\partial X$. But in the generality of our situation, $\rho$ only defines a quasi-metric. However for $s \in (0,s_0)$, where $-s^2_0$ is the asymptotic upper curvature bound of $X$, one has a metric on $\partial X$, say $\rho_s$, bi-Lipschitz to $\rho^s$. In all our results, the Hausdorff dimensions or Hausdorff outer measures are with respect to $\rho_s$.

\medskip

The reader is referred to section $2$ for  any unexplained notations and terminologies.

\medskip

Our first result gives an upper bound on the Hausdorff dimensions of the sets defined in (\ref{Ebeta}) and (\ref{Ebetainf}) for Poisson integrals of complex measures:
\begin{theorem} \label{poisson_thm}
Let $X$ be a non-positively curved Harmonic manifold of purely exponential volume growth with mean curvature of horospheres $h>0$. Assume $\beta \in [0,h]$ and $\mu$ to be a complex measure on $\partial X$. Then 
\begin{equation*}
dim_{\mathcal{H}} E_\beta(P[\mu]) \le (h-\beta)/s \:, \text{ and }
\mathcal{H}^{(h-\beta)/s}\left(E^\infty_\beta(P[\mu])\right) = 0 \:.
\end{equation*}
\end{theorem}
In fact, the bounds in Theorem \ref{poisson_thm} are sharp. This is illustrated by the following result:
\begin{theorem} \label{poisson_sharp_thm}
Let $X$ be as in the statement of Theorem \ref{poisson_thm}. Assume $\beta \in [0,h)$ and $E \subset \partial X$ with $\mathcal{H}^{(h-\beta)/s}(E)=0$. Then there exists a non-negative integrable function $f$ on $\partial X$ (with respect to the visibility measure $\lambda_o$) such that $E \subset E^\infty_\beta\left(P[f]\right)$\:.
\end{theorem}

In the classical Euclidean setting, analogues of Theorem \ref{poisson_thm} were obtained by Armitage \cite[Theorem 4 with Corollary of Theorem 2]{A} for the half-space and by Bayart-Heurteaux \cite[Theorem 1 or 3]{BH} for the unit ball. In the case of $\mathbb H^n(-1)$, the $n$-dimensional real Hyperbolic ball with constant sectional curvature equal to $-1$, analogues of Theorems \ref{poisson_thm} and \ref{poisson_sharp_thm} were recently obtained by Hirata \cite[Theorems 3 and 5]{H}.

\medskip

Now in $\mathbb H^n(-1)$, let $\mu$ be a non-negative Borel measure such that its Green potential $G[\mu]$ is well-defined. Then $G[\mu]$ has radial limit $0$ at almost all points on the boundary, whereas its boundary behavior along other non-tangential directions need not be nice \cite[Theorem 9.4.1]{St}. Similar results for the unit ball in $\C^n$ can be found in the works of Ullrich \cite{U}. These results regarding well-behaved radial limits of Green potentials on a full measure subset of the boundary intrigue us to consider the same problem of exceptional sets for Green potentials. Then one notes that Green potentials are just special examples of positive superharmonic functions. Finally motivated by \cite{H}, we endeavour to obtain results similar to Theorems \ref{poisson_thm} and \ref{poisson_sharp_thm} for the class of positive superharmonic functions.

\medskip

As a first step of analyzing subharmonic (or superharmonic) functions, we obtain their Riesz decomposition, which may be a result of independent interest and seems to be new even for the case of Damek-Ricci spaces. For the relevant definitions in the following statement, the reader is referred to section $4$. 
\begin{theorem} \label{riesz_decom}
Let $X$ be as in the statement of Theorem \ref{poisson_thm}. Let $f$ be a subharmonic function on $X$ such that it has a harmonic majorant. Then
\begin{equation*}
f(x) = F_f(x) - \int_X G(x,y) d\mu_f(y) \:,\text{ for all } x \in X \:,
\end{equation*}
where $F_f$ and $\mu_f$ are the least harmonic majorant and the Riesz measure of $f$ respectively.
\end{theorem}

Then Theorem \ref{riesz_decom} motivates us back to our problem of determining the size of exceptional sets for Green potentials. As the Riesz measure of a subharmonic (or superharmonic) function is a Radon measure, we are naturally interested to look at an analogue of Theorem \ref{poisson_thm} for Green potentials of Radon measures on $X$:
\begin{theorem} \label{Green_thm}
Let $X$ be a non-positively curved Harmonic manifold of purely exponential volume growth with mean curvature of horospheres $h>n-2$ and sectional curvature $K_X \ge -b^2$, for some $b > 0$. Let $\beta \in [0, h-n+2)$ and $\mu$ be a Radon measure on $X$ whose Green potential $G[\mu]$ is well-defined. Then for $b':=\max\{2b,1\}$, we have 
\begin{equation*}
dim_{\mathcal{H}}E_\beta\left(G[\mu]\right) \le b'\left(h-\beta\right)/s\:,\text{ and }
\mathcal{H}^{b'(h-\beta)/s}\left(E^\infty_\beta\left(G[\mu]\right)\right) =0 \:.
\end{equation*}
\end{theorem}

We then have the following analogue of Theorem \ref{poisson_sharp_thm} for Green potentials.
\begin{theorem}\label{green_sharp_thm}
Let $X$ be a non-positively curved Harmonic manifold of purely exponential volume growth with mean curvature of horospheres $h>n-2$. Let $\beta \in [0, h-n+2)$ and $E \subset \partial X$ with $\mathcal{H}^{(h-\beta)/s}(E)=0$. Then there exists a Green potential $u$ on $X$ such that $E \subset E^\infty_\beta(u)$.
\end{theorem}
\begin{remark}
\begin{enumerate}
\item The condition $h>n-2$ is naturally posed due to the behavior of the Green function near its pole. Moreover, for any $\varepsilon>0$, all non-compact Harmonic manifolds with sectional curvature $K_X \le -{(1+\varepsilon)}^2{\left(\frac{n-2}{n-1}\right)}^2$, satisfy this property. The last statement follows from  the fact that the mean curvature of horospheres is obtained as the Laplacian of the Busemann functions and an application of the Hessian comparison theorem. 

\medskip

\item Comparing with Theorems $4$ and $6$ of \cite{H}, it follows that the Hausdorff dimension appearing in Theorem \ref{green_sharp_thm} is optimal. Then we note the gap in the corresponding Hausdorff dimensions in Theorem \ref{Green_thm} and Theorem \ref{green_sharp_thm} (when $b>1/2$), whereas in the case of Theorem \ref{poisson_thm}, the upper bound was shown to be sharp by Theorem \ref{poisson_sharp_thm}.  As it will be apparent from our arguments, the reason is two-fold. Firstly, unlike in the case of the Poisson kernel, the Green function has its singularity in the interior of the space. Hence while trying to compute the Hausdorff dimension of the exceptional set on the boundary, we have to project the analysis done in the interior to the boundary. This is where the geometric ingredient of variable curvature comes into play. Then for $b>1/2$, due to the pinching condition $-b^2 \le K_X \le 0$, we get a gap in the corresponding Hausdorff dimensions. 
\end{enumerate}
\end{remark}

Finally as a consequence of the above results we obtain our main result:
\begin{theorem} \label{superh_thm}
Let $X$ be as in the statement of Theorem \ref{Green_thm}. Let $u$ be a positive superharmonic function on $X$  and $\beta \in [0,h-n+2)$. Then for $b':=\max\{2b,1\}$, we have
\begin{equation*}
dim_{\mathcal{H}}E_\beta(u) \le b'\left(h-\beta\right)/s\:,\text{ and }
\mathcal{H}^{b'\left(h-\beta\right)/s}\left(E^\infty_\beta(u)\right) =0 \:.
\end{equation*}
Conversely, let $X$ be as in the statement of Theorem \ref{green_sharp_thm}. Then for $\beta \in [0, h-n+2)$ and $E \subset \partial X$ with $\mathcal{H}^{(h-\beta)/s}(E)=0$,  there exists a positive superharmonic function $u$ on $X$ such that $E \subset E^\infty_\beta(u)$.
\end{theorem}

In the proofs of Theorems \ref{poisson_thm}, \ref{poisson_sharp_thm}, \ref{Green_thm} and \ref{green_sharp_thm}, we follow the general outline of the arguments in \cite{H} but unlike in the case of $\mathbb H^n(-1)$, the boundary of a non-positively curved Harmonic manifold of purely exponential volume growth is not sufficiently regular and hence our arguments take a substantial detour by estimating global geometric quantities. Unlike in the case of $\mathbb H^n(-1)$, non-constant curvature makes it hard to get sharp estimates on Riemannian angles and the diameters of `shadows' of balls. Then in order to get workable estimates of the above, one has to rely upon comparison principles afforded by the pinching condition on the sectional curvature for `small' balls and the shadow lemma of Gromov hyperbolic spaces for `large' balls. All of this ultimately results in the gap in the corresponding Hausdorff dimensions for Theorems \ref{Green_thm} and \ref{green_sharp_thm}, which can be viewed as a distinct feature of variable pinched non-positive curvature.

\medskip

The arguments in the proof of Theorem \ref{riesz_decom} follow the classical steps presented as in \cite{St} and \cite{U} but unlike in their case, our space need not be a Riemannian symmetric space and hence their approach of $M\ddot{o}bius$ group invariant potential theory breaks down. Instead, we look at a geometric manifestation of convolution and work our way through to obtain results similar to that of the homogeneous setup. 

\medskip

This paper is organized as follows. In section $2$, we recall the required preliminaries and fix our notations. In section $3$, we present our results on the Poisson integrals: Theorems \ref{poisson_thm} and \ref{poisson_sharp_thm}. In section $4$ we prove the Riesz decomposition theorem for subharmonic functions: Theorem \ref{riesz_decom}. In section $5$, the results for Green potentials: Theorems \ref{Green_thm} and \ref{green_sharp_thm}, are proved. Section $6$ consists of the proof of Theorem \ref{superh_thm}. 
 
\section{Preliminaries}
Throughout this article, $C(.)$ will be used to denote positive constants whose value may vary at each occurence, with dependence on parameters or geometric quantities made explicit inside the round bracket. When required, enumerated constants $C_1, C_2, \dots$ will be used to specify fixed constants. 

\medskip

Let $f_1$ and $f_2$ be two positive functions. Then the notation $f_1 \asymp f_2$ will imply that there exists $C>1$ such that $(1/C) f_1 \le f_2 \le C f_1$. Also $f_1 \gtrsim f_2$ (respectively, $f_1 \lesssim f_2$) will imply that there exists $C>0$ such that $f_1 \ge C f_2$ (respectively, $f_1 \le C f_2$). The indicator function of a set $A$ will be denoted by $\chi_A$.

\subsection{Gromov Hyperbolic Spaces}
In this subsection we recall briefly some  basic facts and definitions related to Gromov hyperbolic spaces. For more details, we refer to \cite{Bridson}.

\medskip

A {\it geodesic} in a metric space $X$ is an isometric embedding $\gamma : I \subset \mathbb{R} \to X$ of an interval into $X$. A metric space $X$ is said to be a {\it geodesic metric space} if any two points in $X$ can be joined by a geodesic. A geodesic metric space $X$ is called {\it Gromov hyperbolic} if there exists a $\delta \ge 0$ such that every geodesic triangle in $X$ is $\delta$-thin, that is, each side is contained in the $\delta$-neighbourhood of the union of the other two sides. This $\delta$ is called the Gromov hyperbolicity constant.

\medskip

For a Gromov hyperbolic space $X$, its {\it boundary at infinity} $\partial X$ is defined to be the set of equivalence classes of geodesic rays in $X$. Here a geodesic ray is an isometric embedding $\gamma : [0,\infty) \to X$ of a closed half-line into $X$, and two geodesic rays $\gamma, \tilde{\gamma}$ are said to be equivalent if the set $\{ d(\gamma(t), \tilde{\gamma}(t)) \ | \ t \geq 0 \}$ is bounded. The equivalence class of a geodesic ray $\gamma$ is denoted by $\gamma(\infty) \in \partial X$.  

\medskip

A metric space is said to be {\it proper} if closed and bounded balls in the space 
are compact. Let $X$ be a proper, geodesic, Gromov hyperbolic space. There is a natural topology on $\overline{X} := X \cup \partial X$, called the {\it cone topology} such that $\overline{X}$ is a compact metrizable space which is a compactification of $X$. In this case, for every geodesic ray $\gamma$,  $\gamma(t) \to \gamma(\infty) \in \partial X$ as $t \to \infty$, and for any $x \in X,\: \xi \in \partial X$ there exists a geodesic ray  $\gamma$ such that $\gamma(0) = x, \gamma(\infty) = \xi$.

\medskip

For $x,y,z \in X$, the Gromov product of $y,z$ with respect to $x$ is defined by,
 \begin{equation} \label{gromov_product}
 (y|z)_x := \frac{1}{2} \left(d(x,y)+d(x,z)-d(y,z)\right) \:.
 \end{equation}

\medskip

If the space $X$ is in addition $CAT(0)$ then for any $x \in X$, the Gromov product ${(\cdot|\cdot)}_x$\:, extends continuously to $\partial X \times \partial X$ (see \cite{B}) and hence we define: 
\begin{equation*}
 (\xi|\eta)_x := \displaystyle\lim_{\substack{y \to \xi \\ z \to \eta}} (y|z)_x \:.
 \end{equation*}
We note that $(\xi|\eta)_x = +\infty$ if and only if $\xi = \eta \in \partial X$. Moreover the above boundary continuity of the Gromov product results in the boundary continuity of the Busemann function defined in (\ref{busemann}).

\subsection{Harmonic Manifolds}
 In this subsection we discuss the required preliminaries on Harmonic manifolds. The materials covered here can be found in \cite{BKP}. 
 
\medskip

Let $X$ be a non-compact harmonic manifold of purely exponential volume growth, with origin $o \in X$. By purely exponential volume growth, it is meant that there exists $h > 0$ such that for all $R > 1$, the volume of metric ball $B(x, R)$ of center $x \in X$ and radius $R$ satisfies
\begin{equation*}
vol(B(x,R)) \asymp e^{hR} \:.
\end{equation*}
In our case, it turns out that the constant $h > 0$ agrees with the mean
curvature of the horospheres. 

\medskip

On Harmonic manifolds, the harmonic functions satisfy the usual mean value property on balls and spheres.

\medskip

For any $v \in  T^1_x X$ and $r > 0$, let $A(v,r)$ denote the Jacobian of the map $v \mapsto \exp_x(rv)$. The definition of a harmonic manifold which has been given in the Introduction is equivalent (\cite[p. 224]{Willmore}) to the fact that this Jacobian is solely a function of the radius, that is, there is a function $A$ on $(0,\infty)$, such that $A(v,r) = A(r)$ for all $v \in T^1X$. This function $A$ is called the {\it density function} of $X$. $A$ satisfies the following asymptotics:
\begin{equation} \label{jacobian_estimate}
A(r) \asymp \begin{cases}
             r^{n-1} & \text{ if } 0<r\le 1 \\
             e^{hr} & \text{ if } r>1 \:,
             \end{cases}
\end{equation}

\medskip

In \cite{Kn12}, it was shown that for $X$, a simply connected non-compact harmonic manifold of purely exponential
volume growth with respect to a fixed basepoint $o \in X$, the condition of purely exponential volume growth is equivalent to either of the following conditions:
\begin{enumerate}
\item $X$ is Gromov hyperbolic.
\item $X$ has rank one.
\item The geodesic flow of $X$ is Anosov with respect to the Sasaki metric.
\end{enumerate}

\medskip

Moreover, the Gromov boundary coincides with the visibility boundary $\partial X$ introduced in \cite{Eberlein}. This last fact follows from the work in \cite{KP16}. 

\medskip
One has a family of measures on $\partial X$ called the visibility measures $\{\lambda_x\}_{x \in X}$. For $x \in X$, let $\theta_x$ denote the normalized canonical measure on $T^1_x X$ (the unit tangent space at $x$), induced by the Riemannian metric and then the visibility measure $\lambda_x$ is obtained as the push-forward of $\theta_x$ to the boundary $\partial X$ under the radial projection. The visibility measures $\lambda_x$ are pairwise absolutely continuous. For $(x, \xi) \in X \times \partial X$, the Poisson kernel is obtained as the following Radon-Nykodym derivative:
\begin{equation*}
P(x, \xi) = e^{-hB_\xi(x)} = \frac{d \lambda_x}{d \lambda_o}(\xi) \:.
\end{equation*}
As a consequence of the above identity one has that $P[\lambda_o] \equiv 1$.

\medskip

The following is the Martin representation formula for positive harmonic functions on $X$, which is a consequence of \cite[Corollary 5.13]{KL}:
\begin{lemma} \label{martinrep} 
Let $u$ be a positive harmonic function on $X$. Then there is a unique, finite, positive Borel measure $\mu$ on $\partial X$ such that $u=P[\mu]$.
\end{lemma}
Next we introduce the notion of radial functions. For $x \in X$, let $d_x$ denote the distance function with respect to the point $x$, that is, $d_x(y):=d(x, y)$. A function $f$ on X is called {\it radial} around a point $x \in X$ if $f$ is constant on geodesic spheres centered at $x$. Then note that for a function $f$ radial around a point $x \in X$, we can associate a function $u$ on $\R$ such that $f=u \circ d_x$.

\medskip

If we just say that a function $f$ is radial, then it will be understood that $f$ is radial around $o$, that is, there exists a function $u$ on $\R$ such that $f=u \circ d_o$. For $x \in X$, one has the definition of an $\it x-translate$ of a radial function $f$ as:
\begin{equation} \label{translate}
\tau_xf:= u \circ d_x \:.
\end{equation}

\medskip

Let $\Delta$ denote the Laplace-Beltrami operator associated to the Riemannian metric on $X$. Then one has the following result for Harmonic manifolds:
\begin{lemma} \label{laplacian_commutes_translation}
Let $f \in C^2(X)$ be radial. Then we have for all $x \in X$,
\begin{equation*}
\tau_x ( \Delta f) = \Delta (\tau_x f) \:.
\end{equation*}
\end{lemma}
\begin{proof}
Let $L_R$ denote the radial part of $\Delta$, that is, the differential operator on $(0,\infty)$ defined by,
\begin{equation*}
L_R := \frac{d^2}{dr^2} + \frac{A'(r)}{A(r)} \frac{d}{dr} \:.
\end{equation*}
Let $f=u \circ d_o$, where $u$ is the corresponding function on $\R$. Then by repeated application of Proposition $3.2$ of \cite{BKP}, we get
\begin{equation*}
\tau_x (\Delta f) = \tau_x \left( \left(L_R u\right) \circ d_o \right) = (L_R u) \circ d_x = \Delta (u \circ d_x) = \Delta ( \tau_x f )\:.   
\end{equation*}
\end{proof}
For a measurable function $f$ on $X$ and a measurable function which is radial, say $g$ on $X$, their convolution is defined as
\begin{equation} \label{convolution}
f*g(x):= \int_X f(y) (\tau_x g)(y) dvol(y) \:,
\end{equation} 
whenever the above integral is well-defined.

\medskip

The following Lemma summarizes a few important properties of convolution. Proofs are straightforward consequences of the definition and can also be found in \cite{BKP,PS}.
\begin{lemma} \label{props_of_conv}
(1) If $f$ and $g$ are both measurable radial functions on $X$ then if their convolution is defined at $x \in X$, one has
\begin{equation*}
 (f*g)(x)=(g*f)(x)\:.
\end{equation*}

(2) If $f$ is a measurable function on $X$, $g$ and $h$ are measurable radial functions on $X$ such that the convolutions are defined at $x \in X$, then 
\begin{equation*}
(f*(g*h))(x) = ((f*g)*h)(x) \:.
\end{equation*}

(3) If $f$ and $g$ are two radial functions on $X$ such that their convolution $f*g$ is defined at all points in $X$ then $f*g$ is also a radial function.
\end{lemma}
For $\xi \in \partial X$, the level sets of the Busemann function $B_\xi$ are called horopsheres based at $\xi$. For all $\xi \in \partial X$, the horsopheres based at $\xi$ have the same positive, constant mean curvature $h>0$ and can be obtained as 
\begin{equation} \label{LapBus}
\Delta B_\xi \equiv h \:.
\end{equation}  

\medskip

The following is a version of the Harnack inequality due to Yau in \cite{Y75}:
\begin{lemma}[Harnack-Yau] \label{Harnack_Yau}
Let $X$ be a Hadamard manifold with $-b^2 \le K_X \le 0$. Then there exists a constant $C(b,n)>0$ such that for any open set $\Omega \subset X$ and every positive harmonic function $u : \Omega \to (0,+\infty)$, one has
\begin{equation*}
\|\nabla \log u\| \le C(b,n) \:,\text{ for all } x \in X \text{ with } d(x, \partial \Omega) \ge 1\:.
\end{equation*}
\end{lemma}

\medskip
We next state without proof an easy consequence of Harnack-Yau: 
\begin{lemma} \label{Harnack_Thm}
Let $X$ be as in Lemma \ref{Harnack_Yau} and $\{f_n\}$ be a non-decreasing sequence of harmonic functions on an open connected set $\Omega \subset X$. Then either $f_n(x) \to +\infty$ for all $x \in \Omega$ or that $\{f_n\}$ converges to a harmonic function uniformly on compact subsets of $\Omega$.  
\end{lemma}

\medskip

While working in polar coordinates, we will frequently use the following notation: for $x \in X$ and $v \in T^1_x X$, $\gamma_{x,v}$ is the geodesic such that $\gamma_{x,v}(0)=x$ and $\gamma'_{x,v}(0)=v$. 

\medskip
 
For $f \in C^2(X)$, one has by the Taylor expansion for $x\in X$, $t>0$ sufficiently small:
\begin{equation} \label{infinitesimal_mvp}
\Delta f(x) \frac{t^2}{2n} + C(n) E(t) = \int_{T^1_x X} \left\{f\left(\gamma_{x,v}(t)\right) - f(x)\right\} \: d\theta_x(v)\:, 
\end{equation}
for some constant $C(n)>0$ and a term $E(t)$ which is of order $t^3$\:. 

\medskip

We recall that for an open subset $\Omega \subset X$, an upper semi-continuous function $f: \Omega \to [-\infty,+\infty)$, with $f \not \equiv -\infty$ is subharmonic on $\Omega$ if 
\begin{equation} \label{submvp}
f(x) \le \int_{T^1_x X} f\left(\gamma_{x,v}(r)\right) d\theta_x(v) \:,
\end{equation}
for all $x \in \Omega$ and $r>0$ sufficiently small. It is known that if $f$ is subharmonic on $X$ then (\ref{submvp}) is true for all $r>0$. Moreover, $f$ is locally integrable and bounded above on compact sets. For $f \in C^2(X)$, the above notion of subharmonicity is equivalent to the condition that $\Delta f  \ge 0$. A function $f$ is superharmonic if $-f$ is subharmonic.

\medskip

Now as in our case,
\begin{equation*}
\int_1^{+\infty} \frac{1}{A(r)} \: dr < +\infty \:,
\end{equation*}
we have a positive Green function, which is a radial function defined by
\begin{equation} \label{Green_fn}
G(r) = \frac{1}{C(n)} \int_{r}^{+\infty} \frac{1}{A(s)} \:ds \:,
\end{equation}
for some constant $C(n)>0$. Then (\ref{jacobian_estimate}) yields the following estimates of the Green function:
\begin{equation} \label{green_estimate}
G(r) \asymp \begin{cases}
             \frac{1}{r^{n-2}} & \text{ if } 0<r\le 1 \\
             e^{-hr} & \text{ if } r>1 \:,
             \end{cases}
\end{equation}
upto a positive constant depending only on $n$ and $h$, denoted by $C_1(h,n)$. Then for $x \in X$ the Green function with pole at $x$ is defined by,
\begin{equation*}
G_x(y) := (G \circ d_x)(y)= G(d(x,y)) \:, \text{ for } y \in X\:,
\end{equation*}
and is denoted by $G(x,y)$. Note that it is symmetric in its arguments. The distributional Laplacian of $G_x$ is, 
\begin{equation*}
\Delta G_x = - \delta_x \:.
\end{equation*}
$G_x$ is harmonic on $X \setminus \{x\}$ and superharmonic on $X$. For a non-negative Borel measure $\mu$ on $X$, we say that it has a well-defined Green potential if there exists $x_0 \in X$ such that
\begin{equation*}
G[\mu](x_0) = \int_X G(x_0,y)\: d\mu(y) < +\infty \:.
\end{equation*}
A well-defined Green potential is again a positive superharmonic function.

\medskip
 
If the sectional curvature, $K_X \le -1$, then $\partial X$ is equipped with the visual metric,
 \begin{equation} \label{visual_metric}
 \rho(\xi,\eta):= e^{-{(\xi|\eta)}_o} \:, \text{ for all } \xi,\eta \in \partial X \:.
 \end{equation} 
For $r \in (0,1]$, we have the visual balls with radius $r$ and center $\xi \in \partial X$,
  \begin{equation} \label{visual_ball}
  \mathscr{B}(\xi,r) = \{\eta \in \partial X : \rho(\xi,\eta) < r\} \:.
  \end{equation}
  
\medskip

In the general case of a Harmonic manifold of purely exponential volume growth, although $\rho$ only defines a quasi-metric, the visibility measure $\lambda_o$  satisfies the following estimate for all $\xi \in \partial X$ and for all $r \in (0,1]$ :
  \begin{equation} \label{visual_measure_estimate}
  \lambda_o\left(\mathscr{B}(\xi,r)\right) \le e^{6\delta h} r^h \:,
  \end{equation}
  where $\delta$ is the Gromov hyperbolicity constant. However, one can get a metric by raising $\rho$ to suitable powers. For such spaces, one has the notion of `asymptotic upper curvature bound of X', denoted by $-s^2_0$ (see \cite{BF, S}). It is a critical exponent $s_0 \in (0,+\infty]$ such that for all $s \in (0,s_0)$, $\rho^{s}$ is Lipschitz metrizable, that is, there exists $C_2=C_2(s)>1$ and a metric $\rho_s$ such that
\begin{equation} \label{metric_relation}
\frac{1}{C_2} \rho_s \le \rho^{s} \le C_2 \rho_s \:.
\end{equation} 
For such a fixed $s \in (0,s_0)$, we work with the metric $\rho_s$. By $\mathscr{B}_s(\xi,r)$ we denote a visual ball in the metric $\rho_s$, with center $\xi$ and radius $r$. Then one has for the following containment relations:
\begin{equation} \label{contain_rel1}
\mathscr{B}_s\left(\xi, \frac{r^s}{C_2}\right) \subset \mathscr{B}(\xi,r) \subset \mathscr{B}_s\left(\xi,C_2 r^s\right) 
\end{equation}
and for $C_3=C^{1/s}_2>1$ ,
\begin{equation} \label{contain_rel2}
\mathscr{B}\left(\xi, \frac{r^{1/s}}{C_3}\right) \subset \mathscr{B}_s(\xi,r) \subset \mathscr{B}\left(\xi,C_3 r^{1/s}\right) \:. 
\end{equation}
\medskip

For $\xi \in \partial X$, following the definition of $\gamma_\xi$ mentioned in the introduction, we define the shadow of a ball $B=B(x,r) \subset X$ (viewed from $o$) at $\partial X$ to be the set
\begin{equation*}
\mathcal{O}_o(B) := \{\xi \in \partial X: \gamma_\xi(t) \in B\:, \text{ for some } t >0\} \:.
\end{equation*}
Using the fact that the underlying $X$ is Gromov $\delta$-hyperbolic, one has the following standard `shadow lemma' for balls  with sufficiently large radius:
\begin{lemma} \label{shadow_lemma}
There exists $C_4=C_4(\delta,s)>0$ such that for $r \in  \left(0, \min\left\{\frac{1}{C^s_4},\frac{1}{C_2}\right\}\right)$  and for all $\xi \in \partial X$, we have
\begin{equation*}
\mathcal{B}_s(\xi,r) \subset \mathcal{O}_o\left(B\left(\gamma_\xi\left(\log\left(\frac{1}{C_4\:r^{1/s}}\right)\right),1+\delta\right)\right) \:.
\end{equation*}
\end{lemma} 
\medskip

As mentioned in the Introduction, the sectional curvature of $X$ satisfies $K_X \ge -b^2$ for some $b>0$. If three points in $X$ lie on the same geodesic, then they are called {\it collinear}. For three points $x,y,z$ which are not collinear, we form the geodesic triangle $\triangle$ in $X$ by the geodesic segments $[x, y],\: [y, z],\: [z, x]$. A comparison triangle is a geodesic triangle $\overline{\triangle}$ in $\mathbb{H}^2(-b^2)$ formed by geodesic segments $[\overline{x}, \overline{y}],\: [\overline{y}, \overline{z}],\: [\overline{z}, \overline{x}]$ of the same lengths as those of $\triangle$ (such a triangle exists and is unique up to isometry). Let $\theta(y,z)$ denote the Riemannian angle between the points $y$ and $z$, subtended at $x$. The corresponding angle between $\overline{y}$ and $\overline{z}$ subtended at $\overline{x}$ is called the {\it comparison angle} of $\theta(y,z)$ in $\mathbb{H}^2(-b^2)$ and denoted by $\theta_b(y,z)$. Then by Alexandrov's angle comparison theorem,
\begin{equation} \label{finite_angle_comparison}
\theta_b(y,z) \le \theta(y,z) \:.
\end{equation}

\medskip

Consider the geodesics that join $x$ to $y$ and the one that joins $x$ to $z$. Now extend these geodesics. Then the extended infinite geodesic rays will hit $\partial X$ at two points, say $\xi$ and $\eta$ respectively. Now as points on these geodesics, say $y'$ and $z'$ in $X$ converge to $\xi$ and $\eta$, the comparison angles of $\theta(x',y')$ increase monotonically, and hence their limit exists. We define the comparison angle $\theta_b(\xi, \eta)$ to be this limit and in fact we have,
\begin{equation} \label{infinite_riemannian_angle_bound}
e^{-b(\xi|\eta)_x} = \sin \left(\frac{\theta_b(\xi,\eta)}{2}\right) \le \sin \left(\frac{\theta(\xi,\eta)}{2}\right) \:,
\end{equation}
where $\theta(\xi,\eta)$ is the Riemannian angle between $\xi$ and $\eta$ subtended at $x$.
 
\subsection{Hausdorff Outer Measure and Hausdorff Dimension}
In the setting of a general metric space, we now briefly recall the definitions of Hausdorff dimensions, Hausdorff outer measure and some of their important properties. These can be found in \cite{F}. 

\medskip

Let $(M,d)$ be a metric space. Then for $\varepsilon > 0$, an $\varepsilon$-cover of a set $E \subset M$ is a countable (or finite) collection of sets $\{U_i\}$ with 
\begin{equation*}
0 < diameter\left(U_i\right) \le \varepsilon \text{, for all }i \text{ such that } E \subset \displaystyle\bigcup_{i} U_i \:.
\end{equation*}
For $t \ge 0$, we recall that
\begin{equation*}
\mathcal{H}^t_\varepsilon(E) := \inf \left\{\displaystyle\sum_{i} {\left(diameter\left(U_i\right)\right)}^t : \{U_i\} \text{ is an } \varepsilon\text{-cover of } E\right\}\:.
\end{equation*}
Then the $t$-dimensional Hausdorff outer measure of $E$ is defined by,
\begin{equation*}
\mathcal{H}^t(E) := \displaystyle\lim_{\varepsilon \to 0} \mathcal{H}^t_\varepsilon(E) \:.
\end{equation*}
The above value remains unaltered if one only considers covers consisting of balls.

\medskip

The Hausdorff dimension of $E$ is defined by
\begin{equation*}
dim_{\mathcal{H}}E := \inf \left\{t \ge 0 : \mathcal{H}^t(E) < +\infty\right\} \:.
\end{equation*}
The following properties of Hausdorff dimension and Hausdorff outer measure  will be crucial:
\begin{itemize}
\item {\it Countable stability:} if $\{E_i\}_{i=1}^\infty$ is a countable sequence of sets in $(M,d)$, then 
\begin{equation*}
dim_{\mathcal{H}}\left(\displaystyle\bigcup_{i=1}^\infty E_i\right) =\displaystyle\sup_{i \in \N} \left\{dim_{\mathcal{H}} E_i\right\} \:.
\end{equation*}
\item {\it Non-increasing in dimension:} if $0 < t_1 \le t_2$ then for any $E$, $\mathcal{H}^{t_2}(E) \le \mathcal{H}^{t_1}(E)$\:. 
\end{itemize}

\section{Boundary behavior of Poisson integrals}
For any complex measure $\mu$ on $\partial X$, its Poisson integral $P[\mu]$ is a complex-valued harmonic function on $X$. In this section we will determine the size of the exceptional sets of such Poisson integrals along radial geodesic rays. The key to this analysis is an estimate in terms of a maximal function.
\subsection{Estimates of Maximal Function}
Let $0< \alpha_1 < \alpha_2 \le 1$ and $\xi \in \partial X$. Then for a complex measure $\mu$ on $\partial X$, we consider the following maximal function:
\begin{equation} \label{maximal_fn}
M_{\alpha_1,\alpha_2}[\mu](\xi) := \displaystyle\sup_{\alpha_1 \le r \le \alpha_2} \frac{|\mu|(\mathscr{B}(\xi,r))}{r^h} \:.
\end{equation}
When $d\mu = f d\lambda_o$ for some suitable function $f$ on $\partial X$, we will denote the corresponding maximal function by $M_{\alpha_1,\alpha_2}[f]$. 

\medskip

Next we see an estimate relating the Poisson integral of a complex measure with the maximal function corresponding to the measure. 
\begin{lemma} \label{maximal_fn_lem}
Let $\tau \ge 1,\: 0<\varepsilon \le 1$ and $\mu$ be a complex measure on $\partial X$. Then there exists a constant $C(h) > 0$ such that for all $t > \log(\tau / \varepsilon)$, one has for all $\xi \in \partial X$,
\begin{equation} \label{maximal_fn_ineq}
\left|P[\mu]\left(\gamma_{\xi}(t)\right)\right| \le C(h) \left\{e^{ht}|\mu|\left(\mathscr{B}\left(\xi,\tau e^{-t}\right)\right) + \frac{M_{\tau e^{-t},\varepsilon}[\mu](\xi)}{\tau^h}  + \frac{e^{-ht}}{\varepsilon^{2h}} |\mu|(\partial X) \right\} \:.
\end{equation}
\end{lemma}
\begin{proof}
Fix $\xi \in \partial X$ and $t > \log(\tau / \varepsilon)$. Then note that $\tau e^{-t} < \varepsilon$. Hence there exists a largest non-negative integer $m$ such that
\begin{equation*}
2^m \tau e^{-t} \le \varepsilon \:.
\end{equation*}
Let 
\begin{eqnarray*}
&\mathscr{B}^{(0)}& = \mathscr{B}\left(\xi,\tau e^{-t}\right) \:, \\
&\mathscr{B}^{(j)}& = \mathscr{B}\left(\xi, 2^j \tau e^{-t}\right) \setminus \mathscr{B}\left(\xi, 2^{j-1} \tau e^{-t}\right) ,\:\text{for } 1 \le j \le m \:,\\
&\mathscr{B}^{(m+1)}& = \partial X \setminus \mathscr{B}\left(\xi, 2^m \tau e^{-t}\right) \:.
\end{eqnarray*}
Now 
\begin{equation*}
\left|P[\mu]\left(\gamma_{\xi}(t)\right)\right| \le \displaystyle\sum_{j=0}^{m+1} I_j \:,
\end{equation*}
where 
\begin{equation*}
I_j = \int_{\mathscr{B}^{(j)}} e^{-hB_{\eta}\left(\gamma_{\xi}(t)\right)} \: d|\mu|(\eta) \:,\:\text{for } 0 \le j \le m+1\:.
\end{equation*}
We note that by triangle inequality, for all $\eta \in \partial X$,
\begin{equation*}
B_{\eta}\left(\gamma_{\xi}(t)\right) = \displaystyle\lim_{t' \to \infty} \left(d\left(\gamma_{\xi}(t), \gamma_\eta(t')\right)- d\left(o, \gamma_\eta(t')\right)\right) \ge -d\left(o, \gamma_\xi(t)\right)=-t \:. 
\end{equation*}
Hence, 
\begin{equation*}
I_0 \le \int_{\mathscr{B}^{(0)}} e^{ht} \:d|\mu|(\eta) = e^{ht} \:|\mu|\left(\mathscr{B}\left(\xi,\tau e^{-t}\right)\right) \:. 
\end{equation*}

\medskip

Next we note that Gromov products are monotonically non-decreasing along geodesics, which is a simple consequence of the triangle inequality. Hence in particular, for all $\eta \in \partial X$ such that $\eta \ne \xi$, one has
\begin{equation*}
\displaystyle\lim_{t' \to \infty} \left(\gamma_{\xi}(t)|\gamma_{\eta}(t')\right)_o \le \left(\xi|\eta\right)_o \:.
\end{equation*}
Combining the above with the facts that
\begin{itemize}
\item  $B_{\eta}\left(\gamma_{\xi}(t)\right) = t - 2 \displaystyle\lim_{t' \to \infty} \left(\gamma_{\xi}(t) | \gamma_\eta(t')\right)_o \:,$
\item $e^{-{(\xi|\eta)}_o} \ge 2^{j-1} \tau e^{-t} \text{ when } \eta \in \mathscr{B}^{(j)}\:,\text{ for } 1 \le j \le m\:,$
\end{itemize}
it follows that
\begin{eqnarray*}
I_j &\le & \int_{\mathscr{B}^{(j)}} e^{-ht}\: e^{2h{(\xi|\eta)}_o}\: d|\mu|(\eta) \\
    & \le & \int_{\mathscr{B}^{(j)}} \frac{e^{-ht}}{{\left(2^{j-1} \tau e^{-t}\right)}^{2h}} \:d|\mu|(\eta) \\
    & \le & \frac{|\mu|\left(\mathscr{B}\left(\xi, 2^j \tau e^{-t}\right) \right)}{{\left(2^{j-2} \tau\right)}^h {\left(2^j \tau e^{-t}\right)}^h} \\
    & \le & \frac{1}{{\left(2^{j-2} \tau\right)}^h} \:M_{\tau e^{-t}, \varepsilon}[\mu](\xi) \:.
\end{eqnarray*}
Therefore, there exists $C(h) > 0$ such that,
\begin{equation*}
\displaystyle\sum_{j=1}^m I_j \le \left(\displaystyle\sum_{j=1}^m \frac{1}{{\left(2^{j-2}\right)}^h}\right) \frac{M_{\tau e^{-t}, \varepsilon}[\mu](\xi)}{\tau^h} \le C(h)\:\frac{M_{\tau e^{-t}, \varepsilon}[\mu](\xi)}{\tau^h} \:.
\end{equation*}
Repeating the same argument as above, we get
\begin{equation} \label{maximal_fn_ineq_last_step}
I_{m+1} \le \int_{\mathscr{B}^{(m+1)}} \frac{e^{-ht}}{{\left(2^{m} \tau e^{-t}\right)}^{2h}} \:d|\mu|(\eta) \:.
\end{equation}
Now by the choice of $m$, 
\begin{equation*}
2^m \tau e^{-t} > \frac{\varepsilon}{2} \:.
\end{equation*}
Plugging the above in (\ref{maximal_fn_ineq_last_step}), it follows that
\begin{equation*}
I_{m+1} \le  \int_{\mathscr{B}^{(m+1)}} \frac{2^{2h}\:e^{-ht}}{{\varepsilon}^{2h}} \:d|\mu|(\eta) \le \frac{2^{2h}\:e^{-ht}}{\varepsilon^{2h}} |\mu|\left(\partial X\right) \:.
\end{equation*}
Then summing up the above estimates, we get (\ref{maximal_fn_ineq}).
\end{proof}
Lemma \ref{maximal_fn_lem} has the following consequences.
\begin{corollary} \label{cor1}
Let $0<\varepsilon \le 1$ and $\mu$ be a complex measure on $\partial X$. Then there exists a constant $C(h) > 0$ such that for all $t > \log(1 / \varepsilon)$, one has for all $\xi \in \partial X$,
\begin{equation*} 
\left|P[\mu]\left(\gamma_{\xi}(t)\right)\right| \le C(h) \left\{ 2M_{ e^{-t},\varepsilon}[\mu](\xi)  + \frac{e^{-ht}}{\varepsilon^{2h}} |\mu|(\partial X) \right\} \:.
\end{equation*}
\end{corollary}
\begin{proof}
The Corollary follows by taking $\tau=1$ in Lemma \ref{maximal_fn_lem}.
\end{proof}
\begin{corollary} \label{cor2}
Let $\xi \in \partial X,\:\tau > 1$ and $t>\log(\tau)$. If $f$ is a non-negative measurable function on $\partial X$ such that $f \equiv 1$ on $\mathscr{B}\left(\xi,\tau e^{-t}\right)$ and $f \le 1$ on $\partial X$, then there exists $C_5=C_5(h, \delta)>0$ (where $\delta$ is the Gromov hyperbolicity constant) such that
\begin{equation*}
P[f]\left(\gamma_{\xi}(t)\right) \ge 1 - \frac{C_5}{\tau^h} \:.
\end{equation*}
\end{corollary}
\begin{proof}
Let $t>\log(\tau)$. We consider 
\begin{equation*}
g:= 1-f \:.
\end{equation*}
Then $g$ is a measurable, non-negative function on $\partial X$ such that
\begin{equation*}
g \equiv 0 \text{ on } \mathscr{B}\left(\xi,\tau e^{-t}\right) \text{ and } g \le 1 \text{ on } \partial X\:.
\end{equation*}
Then applying Lemma \ref{maximal_fn_lem}, for $d\mu = g\: d\lambda_o$ and $\varepsilon=1$, we get that there exists $C(h)>0$ such that, 
\begin{eqnarray} \label{cor2_eq}
P[g]\left(\gamma_\xi(t)\right) &\le & C(h) {\left(\frac{1}{\tau}\right)}^h M_{\tau e^{-t},1}[\mu](\xi) \nonumber\\
&\le & C(h) {\left(\frac{1}{\tau}\right)}^h M_{\tau e^{-t},1}[\lambda_o](\xi) 
\end{eqnarray}
Now using (\ref{visual_measure_estimate}) and (\ref{maximal_fn}), it follows that
\begin{equation*}
M_{\tau e^{-t},1}[\lambda_o](\xi) = \displaystyle\sup_{\tau e^{-t} \le r \le 1} \frac{\lambda_o\left(\mathscr{B}(\xi,r)\right)}{r^h} \le e^{6\delta h} \:.
\end{equation*}
Then plugging the above in (\ref{cor2_eq}), one has for some $C(h, \delta)>0$,
\begin{equation*}
P[g]\left(\gamma_\xi(t)\right) \le \frac{C(h,\delta)}{\tau^h} \:. 
\end{equation*}
Thus,
\begin{equation*}
P[f]\left(\gamma_{\xi}(t)\right) = 1 - P[g]\left(\gamma_{\xi}(t)\right) \ge 1 - \frac{C(h,\delta)}{\tau^h} \:. 
\end{equation*}
\end{proof}

\subsection{Upper bound on the Hausdorff dimension}

\begin{proof}[Proof of Theorem \ref{poisson_thm}]
For $L>0$, we set
\begin{equation} \label{defn_EL}
E^L_{\beta}(P[\mu]) :=\left\{\xi \in \partial X : \displaystyle\limsup_{t \to +\infty} e^{-\beta t} \left|P[\mu]\left(\gamma_{\xi}(t)\right)\right| > L\right\} \:.
\end{equation}

\medskip

Our strategy will be to get some useful estimates on the $(h-\beta)/s$\:-dimensional outer Hausdorff measure of the set defined in (\ref{defn_EL}). First we choose and fix $\varepsilon \in (0,1)$ and $\xi \in E^L_{\beta}(P[\mu])$. Then by Corollary \ref{cor1} there exists $C(h) > 0$ such that
\begin{equation*}
C(h)L < \displaystyle\limsup_{t \to +\infty} e^{-\beta t}\: M_{e^{-t}, \varepsilon}[\mu](\xi).
\end{equation*}
Hence, there exists $t_\xi \in (0,+\infty)$ satisfying $e^{-t_{\xi}} \le \varepsilon$ such that
\begin{equation} \label{poisson_pf_eq}
C(h)L < e^{-\beta t_\xi}\: \frac{|\mu|\left(\mathscr{B}\left(\xi, e^{-t_\xi}\right)\right)}{e^{-ht_\xi}} \le  e^{-\beta t_\xi}\: \frac{|\mu|\left(\mathscr{B}_s\left(\xi,C_2\: e^{-st_\xi}\right)\right)}{e^{-ht_\xi}}  \:.
\end{equation}
Now by Vitali 5-covering Lemma, there exist countably many visual balls $\{\mathscr{B}_s\left(\xi_j, r_j\right)\}_{j=1}^\infty$ satisfying (\ref{poisson_pf_eq}) such that
\begin{itemize}
\item $r_j := C_2\:e^{-st_{\xi_j}} \le C_2 \:\varepsilon^s$\:, for all $j \in \N$,
\item $\mathscr{B}_s\left(\xi_j, r_j\right) \cap \mathscr{B}_s\left(\xi_k, r_k\right) = \emptyset$ for all $j \ne k$,
\item $E^L_{\beta}(P[\mu]) \subset \displaystyle\bigcup_{j=1}^\infty \mathscr{B}_s\left(\xi_j, 5r_j\right)$ \:.
\end{itemize} 

Then by (\ref{poisson_pf_eq}), there exists $C(h,\beta,s)>0$ such that
\begin{eqnarray*}
\displaystyle\sum_{j=1}^\infty {\left(diameter\left(\mathscr{B}_s\left(\xi_j, 5r_j\right)\right)\right)}^{(h-\beta)/s}  & \le & \left(\frac{C(h,\beta,s)}{L}\right) \displaystyle\sum_{j=1}^\infty |\mu|\left(\mathscr{B}_s\left(\xi_j, r_j\right)\right) \\
&=&  \left(\frac{C(h,\beta,s)}{L}\right)  |\mu|\left(\bigcup_{j=1}^\infty \mathscr{B}_s\left(\xi_j, r_j\right)\right) \\
&\le &  \left(\frac{C(h, \beta, s)}{L}\right)  |\mu|(\partial X) \:.
\end{eqnarray*}

We note that the constant appearing in the right hand side of the last inequality is independent of the choice of $\varepsilon$ and hence letting $\varepsilon \to 0$, we get that 
\begin{equation} \label{poisson_pf_eq2}
\mathcal{H}^{(h-\beta)/s}\left(E^L_{\beta}(P[\mu])\right) \le \frac{C(h, \beta, s)}{L} |\mu|\left(\partial X\right) < +\infty \:. 
\end{equation}

As $E^\infty_{\beta}(P[\mu]) \subset E^L_{\beta}(P[\mu])$ for all $L>0$, it follows that
\begin{equation*}
\mathcal{H}^{(h-\beta)/s}\left(E^\infty_{\beta}(P[\mu])\right) =0 \:.
\end{equation*}

Finally combining countable stability of the Hausdorff dimension and (\ref{poisson_pf_eq2}) we obtain,
\begin{equation*}
dim_{\mathcal{H}} E_\beta(P[\mu]) = \displaystyle\sup_{m \in \N} \left\{dim_{\mathcal{H}} E^{\frac{1}{m}}_\beta(P[\mu])\right\} \le (h-\beta)/s \:.
\end{equation*}
\end{proof}

\subsection{The sharpness result}

\begin{proof}[Proof of Theorem \ref{poisson_sharp_thm}]
Since $\mathcal{H}^{(h-\beta)/s}(E)=0$, for any $m \in \N$, there exists a covering of $E$ by visual balls $\{\mathscr{B}^{(m,j)}_s\}_{j=1}^\infty$ such that
\begin{equation} \label{poisson_sharp_eq1}
\displaystyle\sum_{j=1}^\infty {\left(diameter\left(\mathscr{B}^{(m,j)}_s\right)\right)}^{(h-\beta)/s} < 2^{-m} \:.
\end{equation}

If $\mathscr{B}_s$ is a visual ball with center $\eta \in \partial X$ and with radius $r$, then for notation convenience, $2\mathscr{B}_s$ will denote the visual ball with the same center $\eta$ and twice the radius, that is, $2r$.

\medskip

Now we define, 
\begin{equation} \label{defn_f_poisson}
f:= \displaystyle\sum_{j,m} m\:{\left(diameter\left(\mathscr{B}^{(m,j)}_s\right)\right)}^{-(\beta/s)}\: \chi_{2\mathscr{B}^{(m,j)}_s} \:.
\end{equation}

Then by (\ref{visual_measure_estimate}) and (\ref{poisson_sharp_eq1}), it follows that for some $C(h,\delta,s)>0$,
\begin{eqnarray*}
\int_{\partial X} f d \lambda_o & \le & \displaystyle\sum_{j,m} m \: {\left(diameter\left(\mathscr{B}^{(m,j)}_s\right)\right)}^{-(\beta/s)} \:\lambda_o\left(2\mathscr{B}^{(m,j)}_s\right) \\
& \le & C(h,\delta,s) \displaystyle\sum_{j,m} m \:{\left(diameter\left(\mathscr{B}^{(m,j)}_s\right)\right)}^{(h-\beta)/s} \\
& < &  C(h,\delta,s) \:\displaystyle\sum_{m=1}^\infty \frac{m}{2^m} \\
& < & +\infty \:.
\end{eqnarray*}

Thus $f d\lambda_o$ defines a finite, positive Borel measure.

\medskip

Now let $\xi \in E$ and fix $m \in \N$. Then there exists $j_m \in \N$ such that $\xi \in \mathscr{B}^{(m,j_m)}_s$. If $r_m$ is the radius of $\mathscr{B}^{(m,j_m)}_s$, then $\mathscr{B}_s(\xi,r_m) \subset 2\mathscr{B}^{(m,j_m)}_s$. Then by Corollary \ref{cor2}, one has
\begin{equation}\label{poisson_sharp_eq2}
P\left[\chi_{2\mathscr{B}^{(m,j_m)}_s}\right](\gamma_{\xi}(t)) \ge P\left[\chi_{\mathscr{B}_s(\xi,r_m)}\right](\gamma_{\xi}(t)) \ge P\left[\chi_{\mathscr{B}\left(\xi,\frac{r^{1/s}_m}{C_3}\right)}\right](\gamma_{\xi}(t)) \ge  \frac{1}{2} \:,  
\end{equation}
whenever (following the statement of Corollary \ref{cor2})
\begin{itemize}
\item $\tau^h > \max \left\{1,\: 2C_5\right\}$,
\item $t > \log(\tau)$,
\item $\tau e^{-t} \le \frac{r^{1/s}_m}{C_3}$\:.
\end{itemize}

Hence choosing $\tau (=\tau(h,\delta))>0$ sufficiently large and setting 
\begin{equation*}
t_m := \log(C_3\tau) + \frac{1}{s}\log\left(\frac{1}{r_m}\right) \:,
\end{equation*}
we have by (\ref{poisson_sharp_eq2}),
\begin{eqnarray*}
P[f](\gamma_\xi(t_m)) & \ge & m {\left(diameter\left(\mathscr{B}^{(m,j_m)}_s\right)\right)}^{-(\beta/s)} P\left[\chi_{2\mathscr{B}^{(m,j_m)}_s}\right]\left(\gamma_{\xi}(t_m)\right) \\
& \ge & m \left(2^{-\left(\frac{\beta}{s} +1\right)} C^{-\beta}_3 \tau^{-\beta} \right)  e^{\beta t_m} \:. 
\end{eqnarray*}

Hence, there exists $C(h,\delta,\beta,s) >0$ such that for all $m \in \N$,
\begin{equation} \label{poisson_sharp_eq3}
e^{-\beta t_m} P[f](\gamma_\xi(t_m)) \ge C(h,\delta,\beta,s)\: m \:.
\end{equation}

Now by (\ref{poisson_sharp_eq1}),
\begin{equation*}
t_m > \log\left(2^{1/s}C_3\tau\right) +  \frac{m}{h-\beta}\log(2) \to +\infty \text{ as } m \to +\infty \:.
\end{equation*}

Hence (\ref{poisson_sharp_eq3}) gives the result.
\end{proof}

\section{Riesz decomposition for subharmonic functions}
\subsection{Riesz measure}
In this subsection our aim would be to prove the existence of a unique Radon measure on $X$ corresponding to a subharmonic function:
\begin{proposition} \label{riesz_meas_prop}
If $f$ is subharmonic on $X$, then there exists a unique Radon measure $\mu_f$ on $X$ such that 
\begin{equation*}
\int_X \psi \:d\mu_f = \int_X f \Delta \psi \:dvol \:, \text{ for all } \psi \in C^2_c(X)\:.
\end{equation*}
\end{proposition} 
\begin{definition}
For a subharmonic function $f$ on $X$, the unique Radon measure $\mu_f$ on $X$ obtained in the conclusion of Proposition \ref{riesz_meas_prop} is called the Riesz measure of $f$. 
\end{definition}
The following lemmas will be important for the proof of Proposition \ref{riesz_meas_prop}.
\begin{lemma} \label{non_decr_mvp}
Let $f$ be a $C^2$ subharmonic function on $X$. Then for all $x \in X$ and for all $0<r_1 \le r_2$, one has
\begin{equation*}
\int_{T^1_x X} f\left(\gamma_{x,v}(r_1)\right) d\theta_x(v) \le \int_{T^1_x X} f\left(\gamma_{x,v}(r_2)\right) d\theta_x(v) \:.
\end{equation*}
\end{lemma}
\begin{proof}
Let $u_1$ and $u_2$ be the harmonic extensions of $f$ on the balls $B(x,r_1)$ and $B(x,r_2)$ respectively. Then by subharmonicity of $f$, it follows that $f \le u_2$ on $B(x,r_2)$ and hence in particular on the sphere $S(x,r_1)$. Hence by the maximum principle, $u_1(x) \le u_2(x)$. Then using the mean value identity of harmonic functions, the result follows. 
\end{proof}
\begin{lemma} \label{Omega_r}
For $r>0$, we define $\Omega_r := \frac{1}{vol\left(B(o,r)\right)} \chi_{B(o,r)}$. Then  for $f \in C^2(X)$ one has for some constant $C(h,n) >0$,
\begin{equation*}
\Delta f(x) = \displaystyle\lim_{r\to 0} \frac{C(h,n)}{r^2} \left\{\left(f*\Omega_r\right)(x) - f(x) \right\}\:.
\end{equation*}
\end{lemma}
\begin{proof}
Integrating the identity (\ref{infinitesimal_mvp}) and then using the estimates of the density function (\ref{jacobian_estimate}) for small $r$ yields the identity,
\begin{equation*}
\Delta f(x) = \displaystyle\lim_{r\to 0} \frac{C(h,n)}{r^2\:vol(B(o,r))} \int_{B(x,r)} \left\{f(y) - f(x) \right\} dvol(y) \:.
\end{equation*}
Now the result follows from the definitions of $\Omega_r$ and the convolution.
\end{proof}
Now we introduce the notion of an approximate identity.
\begin{definition} \label{approx_id}
A sequence of non-negative continuous functions $\{h_j\}_{j=1}^\infty$ is an approximate identity in $L^1(X, dvol)$ if
\begin{itemize}
\item[(i)] $\int_X h_j \:dvol=1$\:, for all $j \in \N$ and

\item[(ii)] ${\displaystyle\lim_{j \to \infty}} \int_{X \setminus B(o,\varepsilon)} h_j\: dvol=0$\:, for all $\varepsilon>0$. 
\end{itemize}
\end{definition}
\begin{remark} \label{ex_of_approx_id}
Let $\{r_j\}_{j=1}^\infty \subset (0,+\infty)$ be a decreasing sequence with $r_j \to 0$ as $j \to +\infty$. For each $j$, let $h_j$ be a non-negative $C^\infty$ radial function on $X$ with support contained in $\{x \in X: r_{j+1} < d(o,x) < r_j\}$ satisfying $\int_{X} h_j \:dvol=1$. Then the sequence $\{h_j\}_{j=1}^\infty$ forms a $C^\infty$-approximate identity. 
\end{remark}
\begin{lemma} \label{approx_id_lemma}
Let $\{h_j\}_{j=1}^\infty$ be a $C^\infty$-approximate identity as defined in remark \ref{ex_of_approx_id}. If $f$ is subharmonic on $X$, then $\{f*h_j\}_{j=1}^\infty$ is a non-increasing sequence of $C^\infty$ subharmonic functions on $X$ satisfying 
\begin{equation} \label{approx_id_eqn}
\left(f*h_j\right)(x) \ge f(x) \text{ and } \displaystyle\lim_{j \to \infty} \left(f*h_j\right)(x)=f(x) \:,
\end{equation}
for all $x \in X$.
\end{lemma}
\begin{proof}
The statement $f*h_j \in C^\infty(X)$ is a simple consequence of the facts that $h_j \in C_c^\infty(X)$, for all $j \in \N$ and $f$ is locally integrable.

\medskip

Let $h_j=u_j\circ d_o$, where $u_j$ is the corresponding function on $\R$. The inequality in (\ref{approx_id_eqn}) follows by integration in polar coordinates,
\begin{eqnarray} \label{approx_id_eq1}
\left(f*h_j\right)(x) &=& \int_{0}^\infty u_j(r)\:A(r) \left(\int_{T^1_x X} f\left(\gamma_{x,v}(r)\right) d\theta_x(v)\right) dr \nonumber\\
& \ge & f(x) \int_{0}^\infty u_j(r)\: A(r)\: dr \nonumber\\
& = & f(x) \:.
\end{eqnarray}

Next we fix $x \in X$ and let $\alpha > f(x)$. By upper semi-continuity of $f$ there exists $r>0$ such that 
\begin{equation*}
f(y) < \alpha \:,\text{ for all } y \in B(x,r) \:.
\end{equation*}
We note that for $r_j < r$,
\begin{equation*}
Supp\left(\tau_x h_j\right) \subset B(x,r) \:.
\end{equation*}
Then 
\begin{equation*}
\left(f*h_j\right)(x) = \int_{B(x,r)} f(y) \left(\tau_x h_j\right)(y)\: dvol(y) \le \alpha \int_X h_j\: dvol = \alpha \:.
\end{equation*}
Hence, 
\begin{equation*}
\displaystyle\limsup_{j \to +\infty} \left(f*h_j\right)(x) \le f(x) \:,
\end{equation*}
which combined with the inequality (\ref{approx_id_eq1}) yields (\ref{approx_id_eqn}).

\medskip

To show that $f*h_j$ is subharmonic for all $j \in \N$, we consider the convolution $(f*h_j)*\Omega_r$, where $\Omega_r$ is as defined in Lemma \ref{Omega_r}. By repeated applications of Lemma \ref{props_of_conv} and by computations similar to that in (\ref{approx_id_eq1}), we get
\begin{equation*}
(f*h_j)*\Omega_r = \left(f*\Omega_r\right)*h_j \ge f * h_j \:.
\end{equation*}
Then in view of Lemma \ref{Omega_r}, it follows that $f*h_j$ is subharmonic for all $j \in \N$. 

\medskip

Finally we show that the sequence is non-increasing. But first we will need an analogue of Lemma \ref{non_decr_mvp} for $f$. Consider $t_2 \ge t_1 >0$.  As $f*h_j$ are $C^\infty$-subharmonic functions with $f*h_j(x) \ge f(x)$ for all $x \in X$, we have by applying Lemma  \ref{non_decr_mvp} to $f*h_j$,
\begin{eqnarray*}
\int_{T^1_x X} f\left(\gamma_{x,v}(t_1)\right)d\theta_x(v) & \le & \int_{T^1_x X} (f*h_j)\left(\gamma_{x,v}(t_1)\right)d\theta_x(v)  \\
& \le & \int_{T^1_x X} (f*h_j)\left(\gamma_{x,v}(t_2)\right)d\theta_x(v) \:.
\end{eqnarray*}
Next using the fact that subharmonic functions are bounded above on compact sets we see that the reverse Fatou lemma is applicable on $\{f*h_j\}_{j=1}^\infty$, which yields
\begin{equation*}
\displaystyle\limsup_{j \to +\infty} \int_{T^1_x X} (f*h_j)\left(\gamma_{x,v}(t_2)\right)d\theta_x(v) \le \int_{T^1_x X} f\left(\gamma_{x,v}(t_2)\right)d\theta_x(v) \:.
\end{equation*}
Combining the last two inequalities one gets the desired analogue of Lemma \ref{non_decr_mvp} for $f$.

\medskip

Now let $m>l$. Since $Supp(h_l)$ is contained in $\{r_{l+1} < d(o,x)< r_l\}$ and $r_{l+1} \ge r_m$, we have
\begin{eqnarray*}
\left(f*h_l\right)(x) &=& \int_{r_{l+1}}^{r_l} u_l(r) A(r) \left(\int_{T^1_x X} f\left(\gamma_{x,v}(r)\right)d\theta_x(v) \right) dr \\
& \ge & \int_{r_{l+1}}^{r_l} u_l(r) A(r) \left(\int_{T^1_x X} f\left(\gamma_{x,v}(r_m)\right)d\theta_x(v) \right) dr \\
&=& \int_{T^1_x X} f\left(\gamma_{x,v}(r_m)\right)d\theta_x(v) \int_X h_l\: dvol \\
&=& \int_{T^1_x X} f\left(\gamma_{x,v}(r_m)\right)d\theta_x(v) \int_X h_m\: dvol \\
& \ge & \int_{r_{m+1}}^{r_m} u_m(r) A(r) \left(\int_{T^1_x X} f\left(\gamma_{x,v}(r)\right)d\theta_x(v) \right) dr \\
&=& \left(f*h_m\right)(x) \:.
\end{eqnarray*}
\end{proof}
\begin{remark} \label{non_decr_mvp_rmk}
\begin{enumerate}
\item The proof of Lemma \ref{approx_id_lemma} shows that Lemma \ref{non_decr_mvp} is true for general subharmonic functions. 

\medskip

\item For a subharmonic function $f$, its restrictions to spheres are integrable, that is, 
\begin{equation} \label{subh_L1}
\int_{T^1_x X} |f(\gamma_{x,v}(r))|\: d\theta_x(v) < +\infty \:, \text{ for all } x \in X \text{ and all } r>0 \:, 
\end{equation}
This is seen as follows. Choose and fix $x \in X$. Then combining the fact that $f$ is locally integrable with the polar decomposition of the volume measure, we get that the function 
\begin{eqnarray*}
(0,+\infty) \to (0,+\infty] \:,\text{ defined by } \\
r \mapsto \int_{T^1_x X} |f(\gamma_{x,v}(r))| \: d\theta_x(v) 
\end{eqnarray*}
is locally integrable with respect to the measure $A(r)dr$. Then as $A(r)dr$ is a regular Borel measure on $(0,+\infty)$, it follows that
\begin{equation} \label{subh_L1_eq1}
\int_{T^1_x X} |f(\gamma_{x,v}(r))| \: d\theta_x(v) < +\infty \:,
\end{equation} 
for almost every $r \in (0,+\infty)$ \:,
with respect to the measure $A(r)dr$. Then as the above measure takes positive values on every non-empty open set in $(0,+\infty)$, it follows that (\ref{subh_L1_eq1}) is true for $r$ belonging to a dense subset of $(0,+\infty)$. Now as $f$ is bounded above on compacts, we get that 
\begin{equation} \label{subh_L1_eq2}
\int_{T^1_x X} f(\gamma_{x,v}(r)) \: d\theta_x(v) > -\infty \:,
\end{equation}
is true for $r$ belonging to a dense subset of $(0,+\infty)$. Then part $(1)$ of this remark yields that (\ref{subh_L1_eq2}) is true for all $r \in (0,+\infty)$. Combining this with the fact that $f$ is bounded above on compacts, (\ref{subh_L1}) follows.
\end{enumerate}

\end{remark}

Now we are in a position to prove Proposition \ref{riesz_meas_prop}.
\begin{proof}[Proof of Proposition \ref{riesz_meas_prop}]
We first show that 
\begin{equation} \label{riesz_meas_pf_eq1}
\int_X f \Delta \psi \:dvol \ge 0\:, \text{ for all } \psi \in C^2_c(X) \text{ with } \psi \ge 0 \:.
\end{equation}
Let $\{h_j\}_{j=1}^\infty$ be a $C^\infty$ approximate identity as defined in remark \ref{ex_of_approx_id}. Set $f_j := f * h_j$\:. Then by Lemma \ref{approx_id_lemma}, $\{f_j\}_{j=1}^\infty$ is a non-increasing sequence of $C^\infty$-subharmonic functions on $X$ that converges to $f$ everywhere on $X$. Now combining the facts that 
\begin{itemize}
\item $\psi \in C^2_c(X)$\:,
\item 
$\left|f_j(x)\right| \le |f(x)|+|f_1(x)| \:,\text{ for all } j \in \N$\:,for all $x \in X$\:,
\item both $f,f_1$ being subharmonic are locally integrable,
\end{itemize}
we note that the Dominated Convergence Theorem is applicable for the sequence of functions $\{f_j \Delta \psi\}_{j=1}^\infty$. Therefore by the Green's identity and the Dominated Convergence Theorem,
\begin{equation*}
\int_X f \Delta \psi \: dvol = \displaystyle\lim_{j \to \infty} \int_X f_j \Delta \psi \:dvol = \displaystyle\lim_{j \to \infty} \int_X \psi \Delta f_j \:dvol  \ge 0 \:.
\end{equation*}

\medskip

Hence, 
\begin{equation*}
L(\psi) := \int_X f \Delta \psi \: dvol
\end{equation*}
defines a positive linear functional on $C^\infty_c(X)$. Then by usual density arguments (following the proof of Theorem $4.6.3$ of \cite{St} for the real hyperbolic ball verbatim) $L$ extends to $C_c(X)$ as a positive linear functional. Now the result follows by the Riesz Representation theorem for positive linear functionals on $C_c(X)$.
\end{proof}

\subsection{Harmonic majorants}
\begin{definition} \label{harmonic_majorant}
\begin{enumerate}
\item A subharmonic function $f$ on $X$ is said to have a harmonic majorant if there exists a harmonic function $h$ on $X$ such that 
\begin{equation*}
f(x) \le h(x)\:,\text{ for all } x \in X\:.
\end{equation*}
\item A harmonic function $h$ on $X$ is said to be the least harmonic majorant of a subharmonic function $f$ on $X$ if
\begin{itemize}
\item[(i)] $h$ is a harmonic majorant of $f$ and
\item[(ii)] $h(x) \le H(x)$ for all $x \in X$, whenever $H$ is a harmonic majorant of $f$\:.
\end{itemize} 
\end{enumerate}
\end{definition}
In this subsection we will prove the following equivalence for existence of a least harmonic majorant in terms of boundedness of integrals over spheres.
\begin{proposition} \label{harmonic_majorant_prop}
Let $f$ be subharmonic on $X$. Then the following are equivalent:
\begin{itemize}
\item[(i)] $f$ has a least harmonic majorant on $X$.
\item[(ii)] $f$ has a harmonic majorant on $X$.
\item[(iii)] for all $x \in X$,\begin{equation*}
\displaystyle\lim_{r \to +\infty} \int_{T^1_x X} f\left(\gamma_{x,v}(r)\right) \:d\theta_x(v) < +\infty \:.
\end{equation*}
\end{itemize}
\end{proposition}
For the proof of Proposition \ref{harmonic_majorant_prop}, we will need the following:
\begin{lemma} \label{harmonic_majorant_lemma}
Let $f$ be a subharmonic function on $X$. Choose and fix $x_0 \in X$. Then for each $r>0$, there exists a harmonic function $f^{(r)}$ on $B(x_0,r)$ such that
\begin{itemize}
\item[(i)] $f(x) \le f^{(r)}(x)$ for all $x \in B(x_0,r)$ and
\item[(ii)] \begin{equation*}
\int_{T^1_{x_0}X} f\left(\gamma_{x_0,v}(r)\right) \:d\theta_{x_0}(v) = f^{(r)}(x_0) \:.
\end{equation*} 
\end{itemize}
Furthermore, if $F$ is harmonic on an open subset $\Omega$ of $X$ with $\overline{B(x_0,r)} \subset \Omega$ and $F(x) \ge f(x)$ for all $x \in \Omega$, then
\begin{itemize}
\item[(iii)] $f^{(r)}(x) \le F(x)$ for all $x \in B(x_0,r)$\:.
\item[(iv)] If $0<r_1<r_2$, then
\begin{equation*}
f^{(r_1)}(x) \le f^{(r_2)}(x)\:,\text{ for all } x \in B(x_0,r_1) \:.
\end{equation*}
\end{itemize}
\end{lemma}
\begin{proof}
Fix $r>0$. Then by upper semi-continuity of $f$, there exists a decreasing sequence $\{f_n\}_{n=1}^\infty$ of continuous functions on $S(x_0,r)= \partial B(x_0,r)$ such that
\begin{equation*}
\displaystyle\lim_{n \to \infty} f_n\left(\gamma_{x_0,v}(r)\right)=f\left(\gamma_{x_0,v}(r)\right)\:,\text{ for all } v \in T^1_{x_0} X\:.
\end{equation*}

\medskip

Now we consider the harmonic extensions of $f_n$ to $B(x_0,r)$, say $F_n$. Then by Harnack-Yau and the mean value property, it follows that the sequence $\{F_n\}_{n=1}^\infty$ is uniformly Cauchy on all compact subsets of $B(x_0,r)$. So the sequence converges (uniformly on compacts) to a harmonic function on $B(x_0,r)$, say $f^{(r)}$. We note that by the maximum principle, $f(x) \le f^{(r)}(x)$ for all $x \in B(x_0,r)$. Moreover it follows from part (2) of remark \ref{non_decr_mvp_rmk} and that $\{f_n\}_{n=1}^\infty$ is a sequence of continuous functions decreasing to $f$, that the Dominated Convergence Theorem is applicable to $\{f_n\}_{n=1}^\infty$. Then by the mean value property and the Dominated Convergence Theorem, it follows that
\begin{eqnarray} \label{harmonic_majorant_lem_eq1}
f^{(r)}(x_0) = \displaystyle\lim_{n \to \infty} F_n(x_0) &=& \displaystyle\lim_{n \to \infty}\int_{T^1_{x_0}X} f_n\left(\gamma_{x_0,v}(r)\right) d\theta_{x_0}(v) \nonumber \\
&=& \int_{T^1_{x_0}X} f\left(\gamma_{x_0,v}(r)\right) d\theta_{x_0}(v) \:.
\end{eqnarray}

\medskip
 
Now for $F$ as in the second part of the statement, we consider
\begin{equation*}
h_n\left(\gamma_{x_0,v}(r)\right) := \min \{f_n\left(\gamma_{x_0,v}(r)\right),F\left(\gamma_{x_0,v}(r)\right)\} \:,\text{ for all } v \in T^1_{x_0} X\:.
\end{equation*}
and $H_n$ to be their corresponding harmonic extensions to $B(x_0,r)$. By the maximum principle, $H_n \le F_n$ on $B(x_0,r)$. But $\{h_n\}_{n=1}^\infty$ is a non-increasing sequence of continuous functions on $S(x_0,r)$ with
\begin{equation*}
\displaystyle\lim_{n \to \infty} h_n\left(\gamma_{x_0,v}(r)\right) = f\left(\gamma_{x_0,v}(r)\right) \:,\text{ for all } v \in T^1_{x_0}X \:.
\end{equation*}
Just as above, an application of Harnack-Yau and the mean value property would yield that the sequence $\{H_n\}_{n=1}^\infty$ converges (uniformly on compacts) to a harmonic function on $B(x_0,r)$. Moreover, by the mean value property, the Dominated Convergence Theorem and (\ref{harmonic_majorant_lem_eq1}),
\begin{equation*}
\displaystyle\lim_{n \to \infty} H_n(x_0) = \int_{T^1_{x_0} X} f\left(\gamma_{x_0,v}(r)\right) d\theta_{x_0}(v) =  \displaystyle\lim_{n \to \infty} F_n(x_0) \:.
\end{equation*}
Hence by the maximum principle,
\begin{equation*}
f^{(r)}(x) = \displaystyle\lim_{n \to \infty} F_n(x) = \displaystyle\lim_{n \to \infty} H_n(x) \:, \text{ for all } x \in B(x_0,r) \:.
\end{equation*}
But again by the maximum principle, 
\begin{equation*}
H_n(x) \le F(x) \:,\text{ for all } n \in \N, \: x \in B(x_0,r) \:,
\end{equation*}
and thus it follows that 
\begin{equation*}
f^{(r)}(x) \le F(x)\:, \text{ for all } x \in B(x_0,r)\:.
\end{equation*}
This proves part $(iii)$ of the Lemma.  Part $(iv)$ follows immediately from $(iii)$. 
\end{proof}
Now we are in a position to prove Proposition \ref{harmonic_majorant_prop}.
\begin{proof}[Proof of Proposition \ref{harmonic_majorant_prop}]
Clearly $(i)$ implies $(ii)$ and $(ii)$ implies $(iii)$. We now suppose that $(iii)$ holds. Choose and fix $x_0 \in X$. Let $\{r_n\}_{n=1}^\infty$ be an increasing sequence of positive real numbers such that $r_n \to +\infty$ as $n \to +\infty$. For each $n \in \N$, let $f^{(n)}$ be the harmonic function on $B(x_0,r_n)$ satisfying the conclusion of Lemma \ref{harmonic_majorant_lemma}. By part $(iv)$ of Lemma \ref{harmonic_majorant_lemma},
\begin{equation*}
f^{(n)}(x) \le f^{(n+1)}(x) \:,\text{ for all } x \in B(x_0,r_n) \:.
\end{equation*} 
Moreover by part $(ii)$ of Lemma \ref{harmonic_majorant_lemma}, 
\begin{equation*}
\displaystyle\lim_{n \to \infty} f^{(n)}(x_0) = \displaystyle\lim_{n \to \infty} \int_{T^1_{x_0}X} f\left(\gamma_{x_0,v}(r_n)\right)\: d\theta_{x_0}(v) < + \infty\:,
\end{equation*}

Therefore $\{f^{(n)}\}_{n=1}^\infty$ is a sequence of non-decreasing harmonic functions with a finite limit at $x_0$. Then by Lemma \ref{Harnack_Thm} it follows that $\{f^{(n)}\}_{n=1}^\infty$ converges uniformly (on compacts) to a harmonic function, say $F_f$. Then by part $(i)$ of Lemma \ref{harmonic_majorant_lemma}, $F_f$ is a harmonic majorant of $f$. In fact, $F_f$ is the least harmonic majorant of $f$, which is seen as follows. Let $F$ be any harmonic majorant of $f$. Then by part $(iii)$ of Lemma \ref{harmonic_majorant_lemma},
\begin{equation*}
f^{(n)}(x) \le F(x)\:,\text{ for all } x \in B(x_0,r_n) \:,\text{ for all } n \in \N\:.
\end{equation*}
Hence, 
\begin{equation*}
F_f(x) \le F(x) \:,\text{ for all } x \in X\:.
\end{equation*}
\end{proof}
Henceforth the least harmonic majorant of a subharmonic function $f$ (when it exists) will be denoted by $F_f$. An immediate consequence of Proposition \ref{harmonic_majorant_prop} is the following:
\begin{corollary} \label{harmonic_majorant_cor}
Let $f \le 0$ be subharmonic on $X$. Then $F_f \equiv 0$ if and only if 
\begin{equation*}
\displaystyle\lim_{r \to +\infty} \int_{T^1_{x}X} f\left(\gamma_{x,v}(r)\right)\: d\theta_x(v) = 0 \:,
\end{equation*}
for all $x \in X$.
\end{corollary}
\begin{proof}
By construction of $F_f$ in the proof of Proposition \ref{harmonic_majorant_prop}, we have for all $x \in X$,
\begin{equation*}
F_f(x) = \displaystyle\lim_{r \to +\infty} \int_{T^1_{x}X} f\left(\gamma_{x,v}(r)\right)\: d\theta_{x}(v) \:.
\end{equation*}
Hence the result follows.
\end{proof}
\subsection{Riesz decomposition}
Having established the notions of the Riesz measure and the least harmonic majorant of a subharmonic function, we now aim to prove the Riesz decomposition theorem.

First we see a couple of lemmas.
\begin{lemma} \label{riesz_decom_lemma1}
Let $h$ be a radial $C^\infty_c$ function on $X$ with $\int_X h \:dvol=0$. Let $v:= - G * h$\:, then $v$ is a radial $C^\infty_c$ function on $X$ with $\Delta v=h$\:.
\end{lemma}
\begin{proof}
Let $h=u\circ d_o$, where $u$ is the corresponding function on $\R$. For fixed $x \in X$, the function $y \mapsto G(x,y)$ is harmonic in $B(o,d(o,x))$ and hence for all $r \in (0,d(o,x))$, by the mean value property,
\begin{equation} \label{riesz_decom_lemma1_eq1}
\int_{T^1_o X} G\left(x,\gamma_{o,v}(r)\right) d\theta_o(v) = G(x,o) = G(d(o,x)) \:.
\end{equation}
We now choose $r>0$ such that $Supp(h) \subset \overline{B(o,r)}$. Then for all $x$ with $d(o,x)>r$, by (\ref{riesz_decom_lemma1_eq1}) and Lemma \ref{props_of_conv} we have
\begin{eqnarray*}
G*h(x) &=& \int_X h(y) G(x,y) \: dvol(y) \\
&=& \int_{0}^{+\infty} u(r) A(r) \left(\int_{T^1_o X}  G\left(x,\gamma_{o,v}(r)\right) d\theta_o(v)\right) \: dr \\
&=& G(d(o,x)) \int_X h\: dvol \\
&=& 0 \:.
\end{eqnarray*}
Thus $Supp(v) \subset \overline{B(o,r)}$\:. Hence $v \in C^\infty_c(X)$ (the regularity follows from the fact that the Green function is locally integrable) and is radial by Lemma \ref{props_of_conv}.

\medskip

Let $\psi \in C^2_c(X)$. Then by Green's identity, symmetry of the Green function and Fubini's theorem, it follows that
\begin{eqnarray*}
\int_X \psi(y) \Delta v(y) \: dvol(y) &=& \int_X v(y) \Delta \psi (y) \: dvol(y) \\
&=&- \int_X \left(\int_X G(x,y) h(x) \: dvol(x)\right) \Delta \psi (y) \: dvol(y)\\
&=& - \int_X h(x) \left(\int_X G(x,y) \Delta \psi (y) \: dvol(y)\right) \: dvol(x) \\
& =& - \int_X h(x) \left(\int_X \Delta G(x,y)  \psi (y) \: dvol(y)\right) \: dvol(x) \\
&=& \int_X h(x) \psi(x) \: dvol(x) \:.
\end{eqnarray*}
Since the above is true for any $\psi \in C^2_c(X)$, we get the result.
\end{proof}
\begin{lemma} \label{riesz_decom_lemma2}
Let $f \le 0$ be a subharmonic function on $X$ satisfying for all $x \in X$,
\begin{equation} \label{riesz_decom_lemma2_eq}
\displaystyle\lim_{r \to +\infty} \int_{T^1_x X} f\left(\gamma_{x,v}(r)\right)\: d\theta_x(v) = 0 \:.
\end{equation}
Then for all $x \in X$,
\begin{equation*}
f(x) = - \int_X G(x,y) \: d\mu_f(y) \:,
\end{equation*}
where $\mu_f$ is the Riesz measure of $f$ \:.
\end{lemma}
\begin{proof}
Let $\{r_j\}_{j=1}^\infty \subset (0,1)$ such that it monotonically decreases to $0$. For each $j \in \N$, let
\begin{eqnarray*}
&&A^1_j := \{y \in X : r_{j+1} < d(o,y) < r_j\} \\
&&A^2_j := \{y \in X : e^{1/r_j} < d(o,y) < e^{1/r_{j+1}}\} \:.
\end{eqnarray*}
For each $j \in \N$ and $k=1,2$, let $h^k_j$ be a non-negative $C^\infty$ radial function with 
\begin{equation*}
Supp \:h^k_j \subset A^k_j\:, \text{ and } \int_X h^k_j \: dvol=1\:.
\end{equation*}
We write $h^k_j = u^k_j \circ d_o$, where $u^k_j$ is the corresponding function on $\R$. We choose and fix $x \in X$. Now since $f$ is subharmonic by Lemma \ref{approx_id_lemma},
\begin{equation} \label{riesz_decom_lemma2_eq1}
f(x) = \displaystyle\lim_{j \to \infty} \left(f* h^1_j\right)(x) = \displaystyle\lim_{j \to \infty} \int_X f(y) \left(\tau_x h^1_j\right)(y) \:dvol(y) \:.
\end{equation}
On the other hand by part $(1)$ of remark \ref{non_decr_mvp_rmk},
\begin{eqnarray*}
\left(f* h^2_j\right)(x) &=& \int_{e^{1/r_j}}^{e^{1/r_{j+1}}} u^2_j(r) A(r) \left(\int_{T^1_x X} f\left(\gamma_{x,v}(r)\right) d\theta_x(v)\right)\:dr \\
& \le & \int_{e^{1/r_j}}^{e^{1/r_{j+1}}} u^2_j(r) A(r) \left(\int_{T^1_x X} f\left(\gamma_{x,v}\left(e^{1/r_{j+1}}\right)\right) d\theta_x(v)\right)\:dr \\
&=& \int_{T^1_x X} f\left(\gamma_{x,v}\left(e^{1/r_{j+1}}\right)\right) d\theta_x(v) \:.
\end{eqnarray*}
Similarly one gets the lower bound to obtain
\begin{equation*}
\int_{T^1_x X} f\left(\gamma_{x,v}\left(e^{1/r_j}\right)\right) d\theta_x(v) \le \left(f* h^2_j\right)(x) \le \int_{T^1_x X} f\left(\gamma_{x,v}\left(e^{1/r_{j+1}}\right)\right) d\theta_x(v) \:.
\end{equation*}
Then in view of (\ref{riesz_decom_lemma2_eq}), we get
\begin{equation} \label{riesz-decom_lemma2_eq2}
\displaystyle\lim_{j \to \infty} \left(f* h^2_j\right)(x) = 0 \:.
\end{equation}
Now for $h_j := h^1_j - h^2_j$, by (\ref{riesz_decom_lemma2_eq1}) and (\ref{riesz-decom_lemma2_eq2}), one has
\begin{equation*}
f(x) = \displaystyle\lim_{j \to \infty} \left(f* h_j\right)(x) \:.
\end{equation*}
We now note that since the Green function is superharmonic on $X$, Lemma \ref{approx_id_lemma} implies that $G* h^1_j$ increases to $G$. Arguments similar to Lemma \ref{approx_id_lemma} also yield that $G*h^2_j$ decreases to $0$. Hence, $G*h_j$ increases to $G$. Now for each $j \in \N$, let $v_j:= - \left(G*h_j\right)$, then by Lemma \ref{riesz_decom_lemma1} we have $v_j \in C^\infty_c(X)$ is radial and $\Delta v_j = h_j$\:.

\medskip

Then by Lemma \ref{laplacian_commutes_translation} and Monotone Convergence Theorem, it follows that
\begin{eqnarray*}
f(x) & = & \displaystyle\lim_{j \to \infty} \int_X f(y) \left(\tau_x h_j\right)(y) \: dvol(y) \\
& = & \displaystyle\lim_{j \to \infty} \int_X f(y) \:\Delta (\tau_x v_j)(y) \: dvol(y) \\
& = & \displaystyle\lim_{j \to \infty}\int_X (\tau_x v_j)(y) \: d\mu_f(y) \\
& = & - \displaystyle\lim_{j \to \infty}\int_X \left(\tau_x \left(G*h_j\right)\right)(y) \: d\mu_f(y) \\
& = & - \int_X \left(\tau_x \: G\right)(y) \: d\mu_f(y) \\
& = & - \int_X G(x,y) \: d\mu_f(y) \:.
\end{eqnarray*}
\end{proof}
\begin{proof}[Proof of Theorem \ref{riesz_decom}]
Let $F_f$ be the least harmonic majorant of $f$. Consider $h := f-F_f$\:. Then $h \le 0$ is a subharmonic function  with the constant zero function as its least harmonic majorant. Then by Corollary \ref{harmonic_majorant_cor}, for all $x \in X$,
\begin{equation*}
\displaystyle\lim_{r \to +\infty} \int_{T^1_x X} h\left(\gamma_{x,v}(r)\right) \:d\theta_x(v) = 0 \:. 
\end{equation*}
Thus by Lemma \ref{riesz_decom_lemma2},
\begin{equation*}
f(x) = F_f(x) - \int_X G(x,y)\: d\mu_h(y) \:.
\end{equation*}
So it is enough to show that $\mu_h = \mu_f$. This is seen as follows. For all $\psi \in C^2_c(X)$, by Green's identity
\begin{eqnarray*}
\int_X \psi \:d\mu_h = \int_X h \Delta \psi \: dvol &=& \int_X f \Delta \psi \: dvol - \int_X F_f \:\Delta \psi \: dvol \\
& = &\int_X f \Delta \psi \: dvol - \int_X \psi \: \Delta F_f   \: dvol \\
& = &\int_X f \Delta \psi \: dvol \\
& = & \int_X \psi \: d\mu_f \:.
\end{eqnarray*}

\end{proof}
\section{Boundary behavior of Green potentials}
\subsection{Upper bound on the Hausdorff dimension} 
In this subsection we will work under the hypothesis of Theorem \ref{Green_thm}. We fix an origin $o \in X$. First, we state the following result regarding a condition imposed on the measure by the well-definedness of its Green potential. It is a simple consequence of the estimates of the Green function (\ref{green_estimate}) and hence the proof is omitted.

\begin{lemma} \label{necess_cond}
Let $\mu$ be a Radon measure on $X$. If $G[\mu]$ is well-defined then 
\begin{equation*}  
\int_X e^{-hd(o,y)} \:d\mu (y) < +\infty \:.
\end{equation*}
\end{lemma}

Now for $\mu$ as in the statement of Theorem \ref{Green_thm} and for any $x \in X$, we write,
\begin{equation*}
G[\mu](x)= \int_X G(x,y)\: d\mu(y) = \int_{B(x,1)} G(x,y)\: d\mu(y) + \int_{X \setminus B(x,1)} G(x,y)\: d\mu(y) \:.
\end{equation*}
We define,
\begin{equation} \label{defn_u1_u2}
u_1(x) := \int_{B(x,1)} G(x,y)\: d\mu(y)\:, 
\text{ and }
u_2(x) := \int_{X\setminus B(x,1)} G(x,y) \:d\mu(y) \:. 
\end{equation}
We first study the boundary behavior of the more well-behaved $u_2$.
\begin{lemma} \label{u2_lemma}
Let $\beta \in [0,h]$ and let $u_2$ be as in (\ref{defn_u1_u2}). Then 
\begin{equation*}
dim_{\mathcal{H}}E_\beta(u_2) \le (h-\beta)/s\:,\text{ and }
\mathcal{H}^{(h-\beta)/s}\left(E^\infty_\beta (u_2)\right) = 0 \:.
\end{equation*}
\end{lemma} 
\begin{proof}
Let $x \in X$. By (\ref{green_estimate}), 
\begin{equation*}
u_2(x) \le C_1 \left\{e^{-hd(o,x)} \mu\left(\{o\}\right) + \tilde{u_2}(x)\right\}\:, 
\end{equation*}
where 
\begin{equation*}
\tilde{u_2}(x) := \int_{X \setminus \left(B(x,1)\cup \{o\}\right)} e^{-hd(x,y)} \:d\mu(y)\:.
\end{equation*}
Then 
\begin{equation*} 
\tilde{u_2}(x) = \int_{X \setminus \left(B(x,1)\cup \{o\}\right)} e^{-hB_y(x)} e^{-hd(o,y)}\: d\mu(y)\:,
\end{equation*}
where 
\begin{equation*}
B_y(x)= d(x,y)-d(o,y) \:.
\end{equation*}
Now rewriting the above in terms of the Gromov product, one has 
\begin{equation*}
B_y(x) = d(o,x) - 2{(x|y)}_o \:.
\end{equation*}
Combining this with the fact that the Gromov product is monotonically non-decreasing along geodesics, it follows that the function, $y \mapsto B_{y}(x)$ is monotonically non-increasing along geodesics. Hence (by denoting $\eta_y \in \partial X$ to be the end-point of the extended infinite geodesic ray that joins $o$ to $y$, for $y \in X$) we get,
\begin{eqnarray} \label{u2_lemma_eq1}
\tilde{u_2}(x) &\le & \int_{X \setminus \left(B(x,1)\cup \{o\}\right)} e^{-hB_{\eta_y}(x)} e^{-hd(o,y)} \: d\mu(y) \nonumber\\
& \le & \int_{X \setminus \{o\}}  e^{-hB_{\eta_y}(x)} e^{-hd(o,y)} \: d\mu(y) \:.
\end{eqnarray} 
Now by Lemma \ref{necess_cond} and the Riesz Representation Theorem for positive linear functionals on the space of continuous functions on $\partial X$, there exists a Radon measure $\tilde{\mu}$ on $\partial X$ such that for all $\varphi$ continuous on $\partial X$, we have
\begin{equation*}
\int_{X \setminus \{o\}} \varphi(\eta_y)\: e^{-hd(o,y)} d\mu(y) = \int_{\partial X} \varphi(\eta)\: d\tilde{\mu}(\eta) \:.
\end{equation*}
Then by the boundary continuity of the Busemann function, it follows that
\begin{equation} \label{u2_lemma_eq2}
\int_{X \setminus \{o\}} e^{-hB_{\eta_y}(x)} e^{-hd(o,y)} d\mu(y) = \int_{\partial X} e^{-hB_\eta(x)} d\tilde{\mu}(\eta) = P[\tilde{\mu}](x) \:.
\end{equation}
Thus combining (\ref{u2_lemma_eq1}) and (\ref{u2_lemma_eq2}), we obtain
\begin{equation*}
u_2(x) \le C_1 \left\{e^{-hd(o,x)} \mu\left(\{o\}\right) + P[\tilde{\mu}](x)\right\} \:.
\end{equation*}
Hence, $E_\beta(u_2) \subset E_\beta\left(P[\tilde{\mu}]\right)$ and $E^\infty_\beta(u_2) \subset E^\infty_\beta\left(P[\tilde{\mu}]\right)$. The result now follows from Theorem \ref{poisson_thm}.
\end{proof}
Now we shift our focus to $u_1$. But first we prove the following geometric estimate.
\begin{lemma} \label{shadow_upper}
Let $B$ be a ball in $X$ with center $z$ and radius $r \in (0,5]$. Then the diameter $d$ of $\mathcal{O}_o(B)$,  satisfies the following upper bound, for some constant $C_6=C_6(b,s)>0$ :
\begin{equation*}
d \le C_6\:r^{s/2b} \:e^{-sd(o,z)} \:,\text{ for } d(o,z) > 6\:.
\end{equation*}
\end{lemma}
\begin{proof}
Let $x \in B$ such that $o,x$ and $z$ are not collinear. We first note that by triangle inequality,
\begin{equation} \label{shadow_upper_eq1}
{(x|z)}_o = \frac{1}{2} \left(d(o,x)+d(o,z)-d(x,z)\right) \ge d(o,z) - d(x,z) \ge d(o,z) -r \:.
\end{equation}
Consider the geodesics that join $o$ to $x$ and the one that joins $o$ to $z$. Extend these geodesics. These extended geodesic rays will hit $\partial X$ at some points, say $\eta$ and $\xi$ respectively. Then
\begin{equation} \label{shadow_upper_eq2}
{(\xi|\eta)}_o - {(x|z)}_o = {(\eta|z)}_x +{(\xi|\eta)}_z \:. 
\end{equation}
Then (\ref{shadow_upper_eq1}) and (\ref{shadow_upper_eq2}) yield,
\begin{equation*}
{(\xi|\eta)}_o \ge  d(o,z) -r +{(\xi|\eta)}_z \:,
\end{equation*}
which in turn gives,
\begin{equation} \label{shadow_upper_eq3}
\rho_s(\xi,\eta) \le C_2\: e^{-s{(\xi|\eta)}_o} \le C_2\: e^{-sd(o,z)}\: e^{sr}\: e^{-s{(\xi|\eta)}_z} \:.
\end{equation}
Now as $r \in (0,5]$, it is enough to obtain an upper bound on $e^{-{(\xi|\eta)}_z}$\:. Let $\alpha$ (respectively  $\theta$) denote the Riemannian angle between $\eta$ and $\xi$ (respectively $\eta$ and $o$) subtended at $z$.   Then $\alpha + \theta = \pi$. Now we claim that 
\begin{equation} \label{shadow_upper_eq4}
\frac{e^{-2b{(o|\eta)}_z}- e^{-2bd(o,z)}}{1-e^{-2bd(o,z)}} \le \sin^2 \left(\frac{\theta}{2}\right) \:.
\end{equation}
The above claim is proved as follows. Let $\gamma$ be the geodesic ray starting from $z$ and hitting $\partial X$ at $\eta$. Now for any $t \in (0, +\infty)$, we consider the geodesic triangle $\triangle(z,o,\gamma(t))$. Then let $\theta_b(t)$ denote the angle corresponding to $\theta$ in the comparison triangle $\overline{\triangle}(z,o,\gamma(t))$ in $\mathbb{H}^2(-b^2)$. By the angle comparison theorem,
\begin{equation*}
\sin\left(\frac{\theta_b(t)}{2}\right) \le \sin\left(\frac{\theta}{2}\right)\:, \text{ for all } t \in (0,+\infty) \:. 
\end{equation*}
Now by the hyperbolic law of cosines,
\begin{equation*}
\sin^2 \left(\frac{\theta_b(t)}{2}\right) = \frac{\cosh bd(o,\gamma(t)) - \cosh b(d(o,z)-d(\gamma(t),z))}{2 \sinh bd(o,z) \sinh bd(\gamma(t),z)} \:.
\end{equation*}
Then 
\begin{eqnarray*}
\displaystyle\lim_{t \to +\infty} \frac{\cosh bd(o,\gamma(t))}{2 \sinh bd(o,z) \sinh bd(\gamma(t),z)} 
&=& \displaystyle\lim_{t \to +\infty} \frac{e^{bd(o,\gamma(t))}+e^{-bd(o,\gamma(t))}}{\left(e^{bd(o,z)}-e^{-bd(o,z)}\right)\left(e^{bd(\gamma(t),z)}- e^{-bd(\gamma(t),z)}\right)} \\
&=& \frac{e^{-2b{(o|\eta)}_z}}{1-e^{-2bd(o,z)}} \:,
\end{eqnarray*}
and 
\begin{eqnarray*}
\displaystyle\lim_{t \to +\infty} \frac{\cosh b(d(o,z)-d(\gamma(t),z))}{2 \sinh bd(o,z) \sinh bd(\gamma(t),z)} 
&=& \displaystyle\lim_{t \to +\infty} \frac{e^{b(d(o,z)-d(\gamma(t),z))}+e^{-b(d(o,z)-d(\gamma(t),z))}}{\left(e^{bd(o,z)}-e^{-bd(o,z)}\right)\left(e^{bd(\gamma(t),z)}- e^{-bd(\gamma(t),z)}\right)} \\
&=& \frac{e^{-2bd(o,z)}}{1-e^{-2bd(o,z)}} \:.
\end{eqnarray*}
Hence, the claim is established. Now as by the triangle inequality,
\begin{equation*}
{(o|\eta)}_z \le d(x,z) \le r \:,
\end{equation*}
plugging the above inequality in (\ref{shadow_upper_eq4}), we get that
\begin{equation} \label{shadow_upper_eq5}
\frac{e^{-2br}- e^{-2bd(o,z)}}{1-e^{-2bd(o,z)}} \le \sin^2 \left(\frac{\theta}{2}\right) \:.
\end{equation}
Then by (\ref{infinite_riemannian_angle_bound}) and (\ref{shadow_upper_eq5}), there exists $C(b)>0$ such that, 
\begin{equation*}
e^{-2b{(\xi|\eta)}_z} \le \sin^2\left(\frac{\alpha}{2}\right) = 1 - \sin^2\left(\frac{\theta}{2}\right) \le 1 - \frac{e^{-2br}- e^{-2bd(o,z)}}{1-e^{-2bd(o,z)}} = \frac{1-e^{-2br}}{1-e^{-2bd(o,z)}} \le C(b) r \:.
\end{equation*}
Hence, there exists $C(b,s)>0$ such that 
\begin{equation} \label{shadow_upper_eq6}
e^{-s{(\xi|\eta)}_z} \le C(b,s)\: r^{s/2b} \:.
\end{equation}
Then plugging (\ref{shadow_upper_eq6}) in (\ref{shadow_upper_eq3}) we get that,
\begin{equation*}
\rho_s(\xi,\eta) \le C(b,s)\:r^{s/2b}\: e^{-sd(o,z)} \:.
\end{equation*}
Now since $x$ was arbitrary, it follows that 
\begin{equation*}
d = \displaystyle\sup_{\eta_1,\eta_2 \in \mathcal{O}_o(B)} \rho_s(\eta_1,\eta_2) \le \displaystyle\sup_{\eta_1,\eta_2 \in \mathcal{O}_o(B)} (\rho_s(\eta_1,\xi) + \rho_s(\xi,\eta_2)\:) \le  2\:C(b,s)\:r^{s/2b}\: e^{-sd(o,z)} \:.
\end{equation*}
\end{proof}
Now we get back to estimating $u_1$. We introduce some new notation. Choose and fix $R>0$ and set for $L>0$,
\begin{equation*}
A_{\beta,R}(L) := \{x \in X \setminus B(o,R): e^{-\beta d(o,x)} u_1(x) > L\} \:.
\end{equation*}
\begin{lemma} \label{u1_lemma1}
Let $L >0,\:\beta \in [0,h-n+2)$ and $u_1$ defined as in (\ref{defn_u1_u2}). Then there exists a countable collection of balls $\{B(x_j,r_j)\}_{j=1}^\infty$ with $d(o,x_j) \ge R$ and $r_j \in (0,5]$ for all $j \in \N$, such that it covers $A_{\beta,R}(L)$ and moreover, one has
\begin{equation*}
\displaystyle\sum_{j=1}^\infty \left\{r_j\: e^{-d(o,x_j)}\right\}^{h-\beta} \le \frac{C(h,n,\beta)}{L}\int_X e^{-hd(o,x)} \: d\mu(x) \:,
\end{equation*}
for some positive constant $C(h,n,\beta)>0$.
\end{lemma}
\begin{proof}
Let 
\begin{equation*}
C_7 := \frac{1}{C_1} \left(1+\frac{2-n}{h-\beta}\right) \:,
\end{equation*}
where $C_1$ is the constant implicit in (\ref{green_estimate}). Note that our hypothesis on $\beta$ implies that $C_7>0$. Choose and fix $x \in X \setminus B(o,R)$. Next we claim that if 
\begin{equation*}
\mu\left(B(x,r)\right) \le C_7\: L\: e^{\beta d(o,x)}\: r^{h-\beta}\:, \text{ holds for all } 0 < r \le 1\:, \text{ then } u_1(x) \le L \: e^{\beta d(o,x)} \:.
\end{equation*}
The above claim is established as follows. By (\ref{green_estimate}), we have
\begin{equation*}
u_1(x) \le C_1 \int_{B(x,1)} {d(x,y)}^{2-n} d\mu(y) \:.
\end{equation*}
Now applying Fubini-Tonelli to the function $(y,r) \mapsto r^{1-n} \chi_{\{(y,r)\: : \: d(x,y)<r\}}$, it follows that
\begin{eqnarray*}
\int_{B(x,1)} {d(x,y)}^{2-n} d\mu(y) &=& \mu\left(B(x,1)\right) + (n-2) \int_{0}^1 r^{1-n} \mu\left(B(x,r)\right) dr \\
& \le & C_7 L\:  e^{\:\beta d(o,x)} \left\{ 1+(n-2)\int_{0}^1 r^{h+1-n-\beta}dr\right\} \\
&=& \frac{L\:  e^{\:\beta d(o,x)}}{C_1} \:.
\end{eqnarray*}
Hence the claim follows. Thus for $x \in A_{\beta,R}(L)$, the above claim ensures the existence of $r_x \in (0,1]$ such that
\begin{equation} \label{u1_lemma1_eq}
e^{-\beta d(o,x)} \mu\left(B(x,r_x)\right) > C_7 L\: r^{h-\beta}_x \:. 
\end{equation}
Then an application of Vitali 5-covering lemma yields a countable, disjoint collection of balls $B(x_j,r_{x_j})$ such that each ball satisfies the inequality (\ref{u1_lemma1_eq}) and for $r_j:=5r_{x_j}$, the balls $B(x_j,r_j)$ cover $A_{\beta,R}(L)$. Then by (\ref{u1_lemma1_eq}), we get
\begin{eqnarray*}
\displaystyle\sum_{j=1}^\infty {\left(r_j\:e^{-d(o,x_j)}\right)}^{h-\beta} &\le & \left(\frac{5^{h-\beta}}{C_7 L}\right) \displaystyle\sum_{j=1}^\infty e^{-hd(o,x_j)} \: \mu\left(B\left(x_j,r_{x_j}\right)\right) \\
&\le & \left(\frac{5^{h-\beta} \:e^h}{C_7 L}\right) \int_X e^{-hd(o,x)}\: d\mu(x) \:.
\end{eqnarray*}
\end{proof}
Now we do the estimates on the boundary. For $\xi \in \partial X$, define
\begin{equation*}
M_{\beta,R}[u_1](\xi):= \displaystyle\sup_{t>R} \: e^{-\beta t}u_1\left(\gamma_\xi(t)\right) \:.
\end{equation*}
\begin{lemma} \label{u1_lemma2}
Let $L>0,\:\beta \in [0,h-n+2)\:,b':=\max\{2b,1\}$ and $u_1$ be as defined in (\ref{defn_u1_u2}). For 
\begin{equation*}
0<\varepsilon < C_6 {\left(\frac{5}{e^6}\right)}^{s/b'} \text{ and } R \ge \log(5)+\left(\frac{b'}{s}\right) \log(C_6/\varepsilon)\:,
\end{equation*}
(where $C_6>0$ is as in the conclusion of Lemma \ref{shadow_upper}), we have,
\begin{equation*}
\mathcal{H}^{b'(h-\beta)/s}_\varepsilon \left(\left\{\xi \in \partial X : M_{\beta,R}[u_1](\xi) > L\right\}\right) \le \frac{C(h,n,\beta,b,s)}{L} \int_X e^{-h d(o,x)} d\mu(x) \:,
\end{equation*}
for some constant $C(h,n,\beta,b,s)>0$.
\end{lemma}
\begin{proof}
Let $\xi \in \partial X$ such that $M_{\beta,R}[u_1](\xi) > L$. Then there 
exists $t_\xi >R$ such that 
\begin{equation*}
e^{-\beta t_\xi}\: u_1\left(\gamma_\xi(t_\xi)\right) > L \:.
\end{equation*}

Then $\gamma_\xi(t_\xi) \in A_{\beta,R}(L)$. Let $\{B\left(x_j,r_j\right)\}_{j=1}^\infty$ be the balls obtained in the conclusion of Lemma \ref{u1_lemma1}. Then 
\begin{equation*}
\left\{\xi \in \partial X : M_{\beta,R}[u_1](\xi) > L\right\} \subset \displaystyle\bigcup_{j=1}^\infty \mathcal{O}_o\left(B\left(x_j,r_j\right)\right) \:.
\end{equation*}
In fact, the diameters of $\mathcal{O}_o\left(B\left(x_j,r_j\right)\right)$ are uniformly bounded by $\varepsilon$. This is seen as follows. We note that the hypothesis on $\varepsilon$ and $R$ imply that $R > 6$ and thus $d(o,x_j) \ge R >6$. Moreover as $r_j \in (0,5]$, Lemma \ref{shadow_upper} is applicable and it yields
\begin{equation} \label{bound_diameter}
diameter\left(\mathcal{O}_o\left(B\left(x_j,r_j\right)\right)\right) \le C_6\: r^{s/b'}_j \: e^{-(s/b') d(o,x_j)} \le C_6 \:5^{s/b'} \:e^{-(s/b')R} \le \varepsilon
\end{equation}  
(the last inequality follows from an elementary computation involving the hypothesis on $R$). Hence by (\ref{bound_diameter}) and Lemma \ref{u1_lemma1}, we get for some constant $C(h,n,\beta,b,s)>0$,
\begin{eqnarray*}
\displaystyle\sum_{j=1}^\infty {\left(diameter\left(\mathcal{O}_o\left(B\left(x_j,r_j\right)\right)\right)\right)}^{b'(h-\beta)/s} &\le & C(h,n,\beta,b,s) \displaystyle\sum_{j=1}^\infty {\left(r_j\: e^{-d(o,x_j)}\right)}^{h-\beta} \\ 
& \le & \frac{C(h,n,\beta ,b,s)}{L} \int_X e^{-hd(o,x)}\: d\mu(x) \:.
\end{eqnarray*}
Hence the result follows.
\end{proof}
We now complete the proof of Theorem \ref{Green_thm}.
\begin{proof}[Proof of Theorem \ref{Green_thm}]
For $L>0$ and $\beta \in [0,h-n+2)$, we define $E^L_\beta(u_1)$ as in the proof of Theorem \ref{poisson_thm}. Then from Lemma \ref{u1_lemma2} and  Lemma \ref{necess_cond}, one has
\begin{equation*}
\mathcal{H}^{b'(h-\beta)/s}\left(E^L_\beta(u_1)\right) \le \frac{C(h,n,\beta,b,s)}{L} \int_X e^{-hd(o,x)} d\mu(x) < +\infty \:.
\end{equation*}
Hence it follows that
\begin{equation*}
\mathcal{H}^{b'(h-\beta)/s}\left(E^\infty_\beta(u_1)\right)=0 \text{ and } dim_{\mathcal{H}} E^{\frac{1}{m}}_\beta(u_1) \le b'(h-\beta)/s\:, \text{ for all } m \in \N\:.
\end{equation*}
Then by the countable stability of the Hausdorff dimension, we get
\begin{equation*}
dim_{\mathcal{H}} E_\beta(u_1) \le b'(h-\beta)/s \:.
\end{equation*}
Now as $E_\beta\left(G[\mu]\right) \subset E_\beta(u_1) \cup E_\beta(u_2)$,  one has from above and Lemma \ref{u2_lemma},
\begin{equation*}
dim_{\mathcal{H}} E_\beta \left(G[\mu]\right) \le \max \{dim_{\mathcal{H}} E_\beta(u_1),dim_{\mathcal{H}} E_\beta(u_2)\} \le b'(h-\beta)/s \:.
\end{equation*}
Next we note that as the Hausdorff outer measure is non-increasing in the dimension,
one has from Lemma \ref{u2_lemma} that
\begin{equation*}
\mathcal{H}^{b'(h-\beta)/s}\left(E^\infty_\beta(u_2)\right) \le \mathcal{H}^{(h-\beta)/s}\left(E^\infty_\beta(u_2)\right) =0 \:.
\end{equation*}
Now as $E^\infty_\beta\left(G[\mu]\right) \subset E^\infty_\beta(u_1) \cup E^\infty_\beta(u_2)$, it follows that
\begin{equation*}
{\mathcal{H}}^{b'(h-\beta)/s} \left(E^\infty_\beta \left(G[\mu]\right)\right) \le {\mathcal{H}}^{b'(h-\beta)/s} \left(E^\infty_\beta (u_1)\right) + {\mathcal{H}}^{b'(h-\beta)/s} \left(E^\infty_\beta (u_2)\right) =0\:.
\end{equation*}
\end{proof}
\begin{remark}
For the class of negatively curved Harmonic manifolds with sectional curvature $-b^2 \le K_X \le -1$, for some $b>1$, one has sharper upper bounds on the Hausdorff dimension of the exceptional sets for the Green potentials of Radon measures:
\begin{equation*}
\frac{b(h-\beta)}{s} \:\:\text{ in stead of }\:\: \frac{2b(h-\beta)}{s} \:.
\end{equation*}

\medskip

First we note that in this case $s=1$ and one can work with the natural metric $\rho$. Then with respect to $\rho$, one has the following estimate on the diameters of shadows of balls:

\medskip

Let $B$ be a ball in $X$ with center $z$ and radius $r \in (0,5]$. Then the diameter $d$ of $\mathcal{O}_o(B)$,  satisfies the following upper bound, for some constant $C(b)>0$ :
\begin{equation*}
d \le C(b) \:r^{1/b} \:e^{-(1/b)\:d(o,z)} \:,\text{ for } d(o,z) > 6\:.
\end{equation*}
The above claim is proved as follows. Let $x,y \in B$ be such that the three points: the origin $o,\:x$ and $y$ are not collinear. Let $\theta(x,y)$ be the Riemannian angle between $x$ and $y$, subtended at the origin $o$. Let $\theta_1(x,y)$ denote the corresponding comparison angle in $\mathbb H^2(-1)$. 
Then by the hyperbolic law of cosine, one has
\begin{equation*}
\sin^2\left(\frac{\theta_1(x,y)}{2}\right) = \frac{\cosh d(x,y)-\cosh \left(\left|d(o,x)-d(o,y)\right|\right)}{2 \sinh d(o,x)\sinh d(o,y)}\:.
\end{equation*}

\medskip

Then as $d(o,x)>1\:,\:d(o,x) \ge d(o,z) - r$ and similarly for $y$, we have
\begin{equation*}
\sinh d(o,x)\sinh d(o,y) \gtrsim e^{2d(o,z)} \:.
\end{equation*}
Also since, $d(x,y) \le 2r \le 10$, it follows that
\begin{equation*}
\cosh d(x,y)-\cosh \left(\left|d(o,x)-d(o,y)\right|\right) \le \cosh 2r -1 \lesssim r^2 \:.
\end{equation*}
Then using the angle comparison theorem, 
we get for some constant $C>0$, the following estimate:
\begin{equation} \label{shadow'_upper_eq1}
\sin\left(\frac{\theta(x,y)}{2}\right) \le \sin\left(\frac{\theta_1(x,y)}{2}\right)  \le C r e^{-\:d(o,z)} \:. 
\end{equation}
Now extend the geodesic rays joining $o$ to $x$ and the one joining $o$ to $y$. These extended geodesic rays hit $\partial X$ at two points, say $\eta_x$ and $\eta_y$ respectively. Then by (\ref{infinite_riemannian_angle_bound}), it follows that
\begin{equation} \label{shadow'_upper_eq2}
e^{-b{(\eta_x|\eta_y)}_o} \le \sin\left(\frac{\theta(\eta_x,\eta_y)}{2}\right) = \sin\left(\frac{\theta(x,y)}{2}\right) \:.
\end{equation} 
Thus by (\ref{shadow'_upper_eq1}) and (\ref{shadow'_upper_eq2}), it follows that there exists $C(b)>0$ such that 
 \begin{equation*}
 \rho(\eta_x,\eta_y) = e^{-{\left(\eta_x|\eta_y\right)}_o} = {\left(e^{-b{\left(\eta_x|\eta_y\right)}_o}\right)}^{1/b}  \le {\left(\sin\left(\frac{\theta(x,y)}{2}\right)\right)}^{1/b}  \le C(b)\: r^{1/b}\: e^{-(1/b)\:d(o,z)} \:. 
 \end{equation*}
The claim now follows by taking supremum over all such $x,y \in B$.

\medskip

Then using the above estimate and following the arguments as in Lemmas \ref{u1_lemma1} and \ref{u1_lemma2}, we get the corresponding Hausdorff dimensions (with respect to $\rho$) in Theorem \ref{Green_thm} to be $b(h-\beta)$\:.
\end{remark}

\subsection{Construction of Green potentials on realizable sets}
In this subsection we work under the hypothesis of Theorem \ref{green_sharp_thm}. Before proceeding with the proof of Theorem \ref{green_sharp_thm}, we first see the following sufficient conditions for a non-negative Borel measure to have a well-defined Green potential, which again is an immediate consequence of (\ref{green_estimate}).
\begin{lemma} \label{suffic_cond}
If $\mu$ is a non-negative Borel measure on $X$ satisfying
\begin{equation*} 
 Supp(\mu) \cap B(o,1) = \emptyset \text{ and }
\int_X e^{-hd(o,y)} \: d\mu(y) <+\infty \:,
\end{equation*} 
then $G[\mu]$ is well-defined.
\end{lemma}

\begin{proof}[Proof of Theorem \ref{green_sharp_thm}]
Since $\mathcal{H}^{(h-\beta)/s}(E)=0$, we have for any $m \in \N$, a covering of $E$ by visual balls $\{\mathscr{B}^{(m,j)}_s\}_{j=1}^\infty$ such that
\begin{equation} \label{green_sharp_eq1}
\displaystyle\sum_{j=1}^\infty {\left(diameter\left(\mathscr{B}^{(m,j)}_s\right)\right)}^{(h-\beta)/s} < 2^{-k_m} \:,
\end{equation}
where $\{k_m\}_{m=1}^\infty$ is a strictly monotonically increasing sequence of positive integers such that
\begin{equation} \label{green_sharp_eq2}
2^{-\frac{k_1s}{h-\beta}} < 2 \min\left\{{\left(\frac{e^{-3(1+\delta)}}{C_4}\right)}^s \:,\frac{1}{C_2} \right\}\:,
\end{equation}
where $C_4>0$ is the positive constant obtained in the conclusion of Lemma \ref{shadow_lemma}.

\medskip

Let $\mathscr{B}^{(m,j)}_s = \mathscr{B}_s\left(\xi^{(m)}_j,r^{(m)}_j\right)$. Then an elementary computation using (\ref{green_sharp_eq1}) and (\ref{green_sharp_eq2}) yields that for all $m,j \in \N$,
\begin{equation*} 
r^{(m)}_j < \min\left\{{\left(\frac{e^{-3(1+\delta)}}{C_4}\right)}^s \:,\frac{1}{C_2} \right\}\:, \:\text{ and hence }\: \log\left(\frac{1}{C_4{\left(r^{(m)}_j\right)}^{1/s}}\right) > 3(1+\delta)\:.
\end{equation*}
Then setting $t^{(m)}_j=\log\left(\frac{1}{C_4{\left(r^{(m)}_j\right)}^{1/s}}\right)$, we get balls $B\left(\gamma_{\xi^{(m)}_j}\left(t^{(m)}_j\right),1+\delta\right)$ whose centers satisfy
\begin{equation} \label{green_sharp_eq3}
d\left(\gamma_{\xi^{(m)}_j}\left(t^{(m)}_j\right),o\right)= t^{(m)}_j >3(1+\delta)\:.
\end{equation}
Now applying Lemma \ref{shadow_lemma}, we get that for all $m,\: j \in \N$,
\begin{equation} \label{contained_in_shadow}
\mathscr{B}^{(m,j)}_s \subset \mathcal{O}_o\left(B\left(\gamma_{\xi^{(m)}_j}\left(t^{(m)}_j\right),1+\delta\right)\right)
\end{equation}
We now set
\begin{equation*}
f := \displaystyle\sum_{j,m} m\:\: e^{\beta t^{(m)}_j} \chi_{B\left(\gamma_{\xi^{(m)}_j}\left(t^{(m)}_j\right),\: 2(1+\delta)\right)} \:.
\end{equation*}
Then as $f$ is defined in terms of indicator functions of balls of radius $2(1+\delta)$ and (by (\ref{green_sharp_eq3}) ) the centers lie outside $\overline{B(o,3(1+\delta))}$, it follows that $f dvol$ defines a non-negative Borel measure on $X$ which is supported away from the unit ball $B(o,1)$. Moreover, there exists a positive constant $C(h,n,\beta,\delta,s)>0$ such that by the triangle inequality, the definition of $t^{(m)}_j$ and (\ref{green_sharp_eq1}) we have,
\begin{eqnarray*}
\int_X e^{-hd(o,x)} f(x) dvol(x) &\le & e^{2h(1+\delta)}\:\displaystyle\sum_{j,m} m\:\: e^{-(h-\beta) t^{(m)}_j} \: vol\left(B\left(\gamma_{\xi^{(m)}_j}\left(t^{(m)}_j\right),2(1+\delta)\right)\right) \\
& \le & C(h,n,\beta ,\delta,s) \displaystyle\sum_{j,m} m\:{\left(r^{(m)}_j\right)}^{(h-\beta)/s}  \\
& \le &  C(h,n,\beta ,\delta,s) \displaystyle\sum_{m=1}^\infty m\:2^{-k_m} \\
& < & +\infty \:. 
\end{eqnarray*}
Then by Lemma \ref{suffic_cond}, $G\left[f\:dvol\right]$ is well-defined. 

\medskip

Now we show that $E \subset E^\infty_\beta(u)$ where $u=G\left[f\:dvol\right]$. Let $\xi \in E$ and choose and fix an $m \in \N$. Now $\xi \in \mathscr{B}^{(m,j_m)}_s$ for some $j_m \in \N$. Then by (\ref{contained_in_shadow}), there exists $x_m \in B\left(\gamma_{\xi^{(m)}_{j_m}}\left(t^{(m)}_{j_m}\right),1+\delta\right)$ such that $x_m$ lies on the geodesic ray $\gamma_\xi$. Since 
\begin{equation*}
B(x_m,1+\delta) \subset B\left(\gamma_{\xi^{(m)}_{j_m}}\left(t^{(m)}_{j_m}\right),2(1+\delta)\right)\:,
\end{equation*}
one has using radiality of the Green function,
\begin{eqnarray*} 
u(x_m) & \ge & m \: e^{\beta t^{(m)}_{j_m}} \int_{B\left(\gamma_{\xi^{(m)}_{j_m}}\left(t^{(m)}_{j_m}\right)\:,\:2(1+\delta)\right)} G(x_m,y) \:dvol(y) \\
& \ge & m \: \left(\frac{e^{\beta d(o,x_m)}}{e^{\beta(1+\delta)}}\right) \int_{B(x_m,1+\delta)} G(x_m,y) \:dvol(y) \\
& \ge & m \: e^{\beta d(o,x_m)}\left(\frac{1}{e^{\beta(1+\delta)}} \int_{B(o,1+\delta)} G(o,y) \:dvol(y) \right) \:.
\end{eqnarray*}
Thus as $m \to +\infty$ we get that $e^{-\beta d(o,x_m)}u(x_m) \to +\infty$. By construction, all the points $x_m$ lie on the geodesic $\gamma_\xi$ and are contained in the balls  $B\left(\gamma_{\xi^{(m)}_{j_m}}\left(t^{(m)}_{j_m}\right),1+\delta\right)$. Then as by (\ref{green_sharp_eq1}),
\begin{equation*}
t^{(m)}_{j_m} \ge \log\left(\frac{2^{1/s}}{C_4}\right) + \left(\frac{k_m}{h-\beta}\right) \log 2 \to +\infty \text{ as } m \to +\infty \:,
\end{equation*}
it follows that $\xi \in E^\infty_\beta(u)$.
\end{proof}

\section{Proof of Theorem \ref{superh_thm}}
As $u$ is a positive superharmonic function, $-u$ is a negative subharmonic function with its least harmonic majorant, $F_{-u} \le 0$. Then by Theorem \ref{riesz_decom} and Lemma \ref{martinrep},
\begin{equation*}
u = P[\mu_1] + G[\mu_2] \:,
\end{equation*}
where $\mu_1$ is a finite, positive Borel measure on $\partial X$ and $\mu_2$ is a Radon measure on $X$. Then as
\begin{equation*}
E_\beta(u) \subset E_\beta\left(P[\mu_1]\right) \cup E_\beta\left(G[\mu_2]\right)\:, 
\end{equation*}
we have from Theorem \ref{poisson_thm}, and Theorem \ref{Green_thm} that
\begin{equation*}
dim_{\mathcal{H}}E_\beta(u) \le \max \left\{dim_{\mathcal{H}}E_\beta\left(P[\mu_1]\right),\:dim_{\mathcal{H}}E_\beta\left(G[\mu_2]\right)\right\} \le b'(h-\beta)/s \:.
\end{equation*}
Similarly one concludes that
\begin{equation*}
\mathcal{H}^{b'(h-\beta)/s}\left(E^\infty_\beta(u)\right) =0 \:.
\end{equation*}

\medskip

The converse part follows from Theorem \ref{green_sharp_thm}.

\medskip

This completes the proof of Theorem \ref{superh_thm}.

\section*{Acknowledgement}
The author would like to thank Prof. Kingshook Biswas for many illuminating discussions and suggestions. The author is supported by a research fellowship from Indian Statistical Institute.

\bibliographystyle{amsplain}

\begin{thebibliography}{amsplain}
\bibitem[AS85]{AS} Anderson, M. and Schoen, R. {\em Positive harmonic functions on complete manifolds of negative curvature}. Ann. Math. 121 (1985), p. 429-461.
\bibitem[Ar81]{A} Armitage, D.H. {\em Normal limits, half-spherical means and boundary measures of half-space Poisson integrals}. Hiroshima Math J. 11(2) (1981), p. 235-246.
\bibitem[BH13]{BH} Bayart, F. and Heurteux, Y. {\em Boundary Multifractal Behaviour for Harmonic Functions in the Ball}. Potential Anal. 38 (2013), p. 499-514.
\bibitem[Be78]{Besse} Besse, A.L. {\em Manifolds all of whose geodesics are closed}. Ergebnisse der Mathematik und inhrer Grenzgebiete, Springer-Verlag Berlin Heidelberg 1978. 
\bibitem[Bi23]{B} Biswas, K. {\em Quasi-metric antipodal spaces and maximal Gromov hyperbolic spaces}. arXiv pre-print (2023): arXiv:2109.03725.
\bibitem[BKP21]{BKP} Biswas, K.; Knieper, G. and Peyerimhoff, N. {\em The Fourier Transform on Harmonic Manifolds of Purely Exponential Volume Growth}. J. Geom. Anal. 31 (2021), p. 126-163.
\bibitem[BF06]{BF} Bonk, M. and Foertsch, T. {\em Asymptotic upper curvature bounds in coarse geometry}. Math. Z. 253 (2006), p. 753-785.
\bibitem[BrH99]{Bridson} Bridson M. R. and  Haefliger, A. {\em Metric spaces of nonpositive curvature}, Grundlehren der
mathematischen Wissenschaften, ISSN 0072-7830; 319, 1999.
\bibitem[EO73]{Eberlein} Eberlein, P. and  O’Neill, B. {\em Visibility manifolds}. Pac. J. Math. 46, 45–109 (1973).
\bibitem[Fa14]{F} Falconer, K. {\em Fractal Geometry, Mathematical Foundations and Applications}. Wiley, 3rd Edition (2014).
\bibitem[Hi21]{H} Hirata, K. {\em Boundary Growth Rates and Exceptional Sets for Superharmonic Functions on the Real Hyperbolic Ball}. J. Geom. Anal. 31 (2021), p. 10586-10602.
\bibitem[KL22]{KL} Kemper, M. and Lokhamp, J. {\em Potential Theory on Gromov Hyperbolic Spaces}. Anal. Geom. Metr. Spaces 2022; 10: 394-431.
\bibitem[Kn12]{Kn12} Knieper, G. {\em New results on noncompact harmonic manifolds}. Comment. Math. Helv. 87, 669–703 (2012).
\bibitem[KP16]{KP16} Knieper, G. and Peyerimhoff, N. {\em Harmonic functions on rank one asymptotically harmonic manifolds}. J. Geom. Anal. 26(2), 750–781 (2016).
\bibitem[Ko69]{Koranyi} Kor\'anyi, A. {\em Boundary behavior of Poisson integrals on symmetric spaces}. Trans. Amer. Math. Soc. 140 (1969), 393-409.
\bibitem[PS15]{PS} Peyerimhoff, N. and Samiou, E. {\em Integral Geometric Properties of Non-compact Harmonic Spaces}. J. Geom. Anal. 25 (2015), p. 122-148.
\bibitem[RR96]{RR} Ramachandran, K. and Ranjan, A. {\em A Pinching constant for Harmonic Manifolds}.  arXiv pre-print (1996): arXiv:dg-ga/9603014.
\bibitem[Sc06]{S} Schroeder, V. {\em Quasi-metric and metric spaces}. Conform. Geom. Dyn. 10 (2006), p. 355-360.
\bibitem[St16]{St} Stoll, M. {\em Harmonic and Subharmonic Function Theory on the Hyperbolic Ball}. London Mathematical Society Lecture Note Series, 431. Cambridge University Press, Cambridge (2016). 
\bibitem[Ul85]{U} Ullrich, D. {\em Radial limits of M-subharmonic functions}. Trans. Amer. Math. Soc. 292, no. 2, (1985), p. 501-518.
\bibitem[Wi93]{Willmore}Willmore, T.J. {\em Riemannian Geometry}. Oxford Science Publications, The Clarendon Press, Oxford University Press, New York (1993)
\bibitem[Ya75]{Y75} Yau, S. T.  {\em Harmonic functions on complete Riemannian manifolds}. Commun. Pure Appl. Math. 28 (1975), p. 201-228.

\end{thebibliography}

\end{document}